\documentclass{article}
\usepackage[utf8]{inputenc}
\usepackage{inputenc,xcolor,amsthm,hyperref, cite, color, url,amsfonts,amsmath,amssymb,caption,multicol,graphicx,fancyhdr,tikz,pgfplots,subcaption,comment,commath,bm}
\usepackage{todonotes}
\usepackage{pdflscape}
\usepgfplotslibrary{groupplots}
\usepackage{nicefrac,xfrac}
\usepackage[toc,page]{appendix}
\usepackage[bottom=3cm, headsep=0cm,headheight=0cm]{geometry}
\usepackage{dsfont}

\usepackage{bbm}
\usepackage{bbm}
\DeclareMathOperator{\E}{\mathbb{E}}

\renewcommand{\l}{\lambda}

\newcommand{\m}{\mu}

\newcommand{\f}[2]{\frac{#1}{#2}}

\usepackage{blindtext}
\usepackage{soul}

\graphicspath{{Figure/}}
\theoremstyle{theorem}
\newtheorem{theorem}{Theorem}[section]
\newtheorem{lemma}[theorem]{Lemma}

\newtheorem{corollary}[theorem]{Corollary}

\theoremstyle{definition}

\newtheorem{remark}[theorem]{Remark}

\newcommand{\pushright}[1]{\ifmeasuring@#1\else\omit\hfill$\displaystyle#1$\fi\ignorespaces}

\usepackage[affil-it]{authblk}
\title{Workload analysis of a two-queue fluid polling model}
\author[a]{S. Kapodistria}
\author[a]{M. Saxena}
\author[a]{O. J. Boxma}
\author[b]{O. Kella}
\affil[a]{Department of Mathematics and Computer Science, Eindhoven University of Technology}
\affil[b]{Department of Statistics and Data Science, The Hebrew University of Jerusalem}
\date{\today}

\newcommand\Rp{\mathbb{R}_+}

\begin{document}
\maketitle

\begin{abstract}
\noindent In this paper, we analyze a two-queue random time-limited Markov modulated polling model. In the first part of the paper, we investigate the fluid version: Fluid arrives at the two queues as two independent flows with deterministic rate. There is a single server that serves both queues at constant speeds. The server spends an exponentially distributed amount of time in each queue. After the completion of such a visit time to one queue, the server instantly switches to the other queue, i.e., there is no switchover time.  

\noindent For this model, we first derive the Laplace-Stieltjes Transform (LST) of the stationary marginal fluid content/workload at each queue. Subsequently, we derive a functional equation for the LST of the two-dimensional workload distribution that leads to a Riemann-Hilbert boundary value problem (BVP). After taking a heavy-traffic limit, and restricting ourselves to the symmetric case, the boundary value problem simplifies and can be solved explicitly. 

\noindent In the second part of the paper, allowing for more general (L\'evy) input processes and server switching policies, we investigate the transient process-limit of the joint workload in heavy traffic. Again solving a BVP, we determine the stationary distribution of the limiting process. We show that, in the symmetric case, this distribution coincides with our earlier solution of the BVP, implying that in this case the two limits (stationarity and heavy traffic) commute.
\end{abstract}

\section{Introduction}

The stationary analysis of multi-dimensional Markov processes associated with queueing models is often quite challenging. Even in the two-dimensional case, the characterization of the stationary distribution of fundamental queueing models (such as the shortest queue routing and the coupled processors \cite{fay,cohen2000boundary}) requires solving  boundary value problems. 
The intrinsic complexity of this analysis has led to the development of asymptotic techniques, studying the stationary distribution in some limiting regime of the model parameters; one prominent example being the {\em heavy-traffic limit}, first introduced by Kingman~\cite{kingman} for the single server queue.
In the heavy-traffic limit, a scaled version of the workload process is shown to have a non-trivial limit, which may serve as an approximation to the non-scaled process.
The methodological contribution in this paper is to combine both approaches: For a specific {\em fluid flow polling model with random time-limited service} (which will be specified later), we first derive the boundary value problem, which characterizes the stationary distribution, but for which no explicit solution is known.
We then formulate the boundary value problem obtained in the heavy-traffic limit, which in the symmetric case  leads to an explicit solution for the two-dimensional stationary distribution (in heavy traffic).
A second contribution of our paper is to investigate the heavy-traffic limit of a generalization of the polling model using process limits, allowing for L\'evy input processes into the queues and a more general switching process for the server.
Following~\cite{kw1992},
instead of directly focusing on the stationary distribution and deriving a functional equation for it, we characterize the entire scaled limit process as a two-dimensional reflected Brownian motion in the positive orthant.
We show that in the earlier special (Markovian) symmetric case the stationary distribution of the heavy-traffic process limit coincides with the heavy-traffic limit of the stationary distribution; thus the heavy-traffic limit and the time limit to stationarity commute. 

Specifically, the model that we consider is a polling model with two queues and a single server that moves between the two queues to provide them with service. The policy that governs the switching is random time-limited (RTL): 
The duration of the service period at any queue is random, having an exponential distribution. All these service periods are independent and the server always remains at a queue until the exponentially distributed time expires, even if that queue is empty and the other is not. 
The input and, when the served queue is not empty, the output processes for both queues are assumed to be deterministic fluid streams (with identical rates).
Our motivation to study this RTL Markov modulated fluid polling model comes from our earlier paper~\cite{Saxena2019}, 
in which the present fluid model emerged as an (asymptotic) approximation of a two-queue RTL polling model with Poisson arrivals and exponential service times. 
In the present paper, we show that even under the simplifying assumptions of fluid flows with constant inflow and outflow rates, and symmetric queues, determining the joint stationary workload distribution still requires solving a complicated boundary value problem (BVP).
In the heavy-traffic limit, we obtain and explicitly solve a BVP which is similar to that studied in
\cite{boxma1993compensation} and belongs to a class of two-dimensional BVP that is being discussed in
\cite{cohen2000boundary} (see also \cite{fay}). 
It is intuitive to recognize that in the \textit{asymmetric} case with different loads on the two queues, the queue dynamics are easier to describe (compared to the symmetric case), as the workloads become independent in heavy traffic, reducing the analysis to that of the two marginals. The heaviest loaded queue reaches the saturation point (and must be scaled) while the other queue remains stable (and needs not be scaled).
For this reason, in the first part of the paper, we focus on the symmetric case: the two queues are entirely symmetric in terms of inflow and outflow rates, as well as the server visiting times.
The symmetry assumption puts us in the most interesting case for the heavy-traffic setting that we consider in this paper; it ensures that the workloads in the two queues are of comparable magnitude in heavy-traffic.
In the second part of the paper, we extend the analysis to both a more general (L\'evy input) model and to the study of the symmetric as well as the asymmetric case.

\paragraph{Related literature.}  
Both fluid queueing models and polling models have received much attention in the literature of stochastic service systems;
we refer to the surveys \cite{Kulkarni,boxma2018fluid} for overviews of the literature on fluid queues,
and to the surveys \cite{boonmeiwinands,BorstBoxma} for similar overviews on polling.

In contrast with the extensive literature on fluid queues and on polling, there are only very few studies focussing on polling systems with fluid input.
Some exceptions are Czerniak and Yechiali \cite{czerniak2009fluid}, Boxma et al. \cite{BIKM}, Remerova et al. \cite{remerova2014random}, and Adan et al. \cite{adan2018optimal}; see also \cite[Section 6]{BorstBoxma}. A recent heavy-traffic analysis of a fluid model with two queues {\em in series} is Koops et al.~\cite{koops2016tandem}.  

Polling models with time-limited service also have not been widely studied. Coffman, Fayolle and Mitrani \cite{CFM}
have analyzed a two-queue polling model with exponential visit periods; in their case (contrary to the service protocol pertaining to the model studied in this paper) the server does not stay at an empty queue.
They determine the probability generating function (PGF) of the joint stationary queue length distribution by solving a Riemann-Carleman BVP.
In a series of papers, Al Hanbali et al. (see, e.g., \cite{AHBO}) consider a polling model with several queues and exponential visit periods.
They relate the PGFs of the number of customers in a queue at the end of the server's visit to that queue and at its beginning. This is used as input for a
numerical scheme to approximate the joint queue length PGF at the server departure instants from the queues.
Further references are provided in \cite{Saxena2017}; that paper, and \cite{Saxena2019}, also present
a perturbation method for obtaining queue length PGFs in time-limited polling models.



\paragraph{Organization of the paper.}
In Section \ref{Chap4:ModelDescription}, we describe the RTL Markov modulated fluid queue under consideration. In Section \ref{Chap4:sec:marginal}, we briefly present the LSTs of the model's {\em marginal} stationary workload distributions and obtain their heavy-traffic limits. Section \ref{Chap4:Jointworkload}~is devoted to a discussion of the {\em joint} workload distribution analysis. In Section \ref{Chap4:HT_Analysis}, restricting ourselves to the symmetric case, we derive an explicit expression for  the LST of the joint stationary workload distribution in heavy traffic by solving a Riemann-Hilbert BVP. Several numerical experiments are performed in Section \ref{Chap4:Numerical_results} in order to get more insight into the model.  Section~\ref{sec:processlimit} is devoted to the computation of the scaled joint stationary distribution of an analogous model with a general L\'evy input, generalizing the results obtained in Sections \ref{Chap4:sec:marginal}--\ref{Chap4:HT_Analysis}. 

\section{Model description and notation}\label{Chap4:ModelDescription}
We consider an RTL Markov modulated fluid polling model with two queues.
In our initial description we will not make any symmetry assumptions between the two queues, to facilitate later presentation and discussions regarding these assumptions in Section~\ref{Chap4:Jointworkload}.
As alluded to in the introduction, our main contributions in the first part of the paper (Sections \ref{Chap4:sec:marginal}--\ref{Chap4:HT_Analysis}) concern the heavy-traffic limit for identical parameter settings for the two queues. Arguably, this is the most interesting case under heavy-traffic conditions, because - as we will make precise later - it ensures that the workload processes of the two queues obey a similar scaling when approaching the heavy-traffic saturation point, and, consequently, exhibit a non-trivial correlation.
To the contrary, in an asymmetric setting, one of the two queues will approach heavy traffic while the other remains bounded. 
 In that case, the two workloads are asymptotically independent and their joint heavy-traffic distribution can be obtained from the marginal scaled limit for the queue with heaviest load and the ordinary (non heavy-traffic) marginal distribution of the lighter loaded queue.

In the asymmetric setting, fluid enters queue $j$ (say $Q_j$) at a constant rate of $\lambda_j>0$, $j = 1, 2$. There is a single server that serves both queues with constant rate $\mu_j>0$, $j = 1, 2$.  A special feature of the model is that the server spends  random amounts of time at each queue, these times are independent of the fluid content levels (workloads at the queues); in particular, when a queue becomes empty, the server will remain at that queue (although not providing any service), even if the other queue is not empty, until the expiration of the random visit time. We denote the length of a generic time interval that the server resides at $Q_j$ by $T_j$, $j = 1, 2$. The periods $T_j$ are exponentially distributed with rate $c_j>0$, $j = 1, 2$. Upon completion of the residence time at $Q_j$, the server instantaneously switches to the other queue $Q_{3-j}$, $j=1,2$, i.e., there is no switch-over time. All residence times are independent.


To analyze the model under consideration, we let $V_j(t)$ denote the workload at $Q_j$ at time $t$, $t\geq0$. Assuming $V_1(0) = u_0$, $V_2(0) = l_0$, and the server being at $Q_1$ at time $0$, we can describe the workload at time  $T_1$ and $T_1 + T_2$ as follows:
\begin{itemize}
\item In the interval $(0, T_1]$ the server is serving $Q_1$, therefore the workload (the fluid content) at $Q_1$ decreases linearly as long as it is positive: $V_1(T_1) = \max\{0,  u_0 + (\l_1 - \m_1) T_1\}$. During this time period, the workload at $Q_2$ increases linearly: $V_2(T_1) = l_0 + \l_2 T_1$.
\item Analogously as explained above, in the interval $(T_1, T_1 + T_2]$, the server is serving $Q_2$, therefore the workload at $Q_2$ decreases linearly as long as it is positive, hence $V_2(T_1 + T_2) = \max\{0,  V_2(T_1) + (\l_2 - \m_2) T_2\}$. The workload at $Q_1$ increases linearly, so $V_1(T_1 + T_2) =  V_1(T_1) + \l_1 T_2$.
\end{itemize}
In stationarity, $\left(V_1(T_1 + T_2) , V_2(T_1 + T_2) \right)$ has the same distribution as $\left(V_1(0) , V_2(0)\right)$.

\paragraph{Stability condition.} For the model under consideration, the stability condition states that both queues must have larger capacities than the respective loads imposed on them:
\begin{equation}\label{Chap4:Stability}
\rho_1<  \frac{ c_2}{  c_1 + c_2}~ \text{ and }~  \rho_2<  \frac{c_1}{c_1 + c_2},
\end{equation}
with $\rho_j=\frac{\lambda_j}{\m_j}$, $j=1,2$, cf.~\cite{Saxena2017}. 

\section{Marginal workload analysis}
\label{Chap4:sec:marginal}
In this section, we first briefly focus on the stationary workload of $Q_1$, and hence by identical arguments also of $Q_2$. Let $V_1(t)$ denote the workload at time $t, t \geq 0,$ and $V_1$ the stationary workload at an arbitrary epoch. From a special case of \cite[Section 5]{kella1992storage}, and also from the analysis performed in \cite{chen1992fluid}, the marginal queue length distributions of the model under consideration are known. We include them here for completeness.

\begin{theorem}\label{thm3.1}
The LST of the (marginal) workload of the first queue in stationarity under the stability condition \eqref{Chap4:Stability} is given by
\begin{align}\label{Chap4:MLQ1}
\E\left(e^{- s  V_1}\right) = \frac{1 + \frac{\rho_1 \mu_1}{c_1 + c_2} s}{1 + \frac{\rho_1 \mu_1}{c_2\left(1 - \f{c_1}{c_2} \f{\rho_1}{1 - \rho_1}\right)} s}.
\end{align}
An equivalent formula holds for the LST of $V_2$ under the stability condition \eqref{Chap4:Stability}.
\end{theorem}

\begin{remark}
From the result of Theorem \ref{thm3.1}, it is evident that, for $\theta_1=\frac{c_1+c_2}{\rho_1\mu_1}$,  $\theta_2=\frac{c_2}{\rho_1\mu_1}\left(1-\frac{c_1}{c_2}\frac{\rho_1}{1-\rho_1}\right)$ and with $\mathcal{E}_\theta\sim\text{exp}(\theta)$, 
\begin{enumerate} 
\item $V_1+\mathcal{E}_{\theta_1}$ is distributed like $\mathcal{E}_{\theta_2}$, with $ V_1,\mathcal{E}_{\theta_1}$ independent.
\item The distribution of $V_1$ is a $(\nicefrac{\theta_2}{\theta_1},1-\nicefrac{\theta_2}{\theta_1})$ mixture of zero and $\mathcal{E}_{\theta_2}$.
\end{enumerate}
\end{remark}

With the result of Theorem \ref{thm3.1}, we can study the behavior of the workload $V_1$ in heavy traffic, i.e., when $\rho_1 \uparrow \f{c_2}{c_1 + c_2}$.

\begin{lemma}\label{Chap4:MarginalHT}  For $\rho_1 \uparrow \f{c_2}{c_1 + c_2}$,
\begin{eqnarray}
\left(\f{c_2}{c_1 + c_2}-\rho_1\right) V_1 \overset{d}{\longrightarrow} Z,
\end{eqnarray}
where $Z$ is an exponentially distributed random variable with mean $\f{c_1c_2 \mu_1}{(c_1 + c_2)^3}$.
\end{lemma}
\begin{proof}
Replacing $s$ by $\left(\f{c_2}{c_1 + c_2}-\rho_1 \right)s$ in \eqref{Chap4:MLQ1} and taking the limit as $\rho_1 \uparrow \f{c_2}{c_1 + c_2}$, yields   
\begin{equation} \label{Chap4:HTMarginal}
\lim\limits_{\rho_1 \uparrow \f{c_2}{c_1 + c_2}}\E\left(e^{-\left(\f{c_2}{c_1 + c_2}-\rho_1\right)sV_1}\right) = \f{1}{ 1 + \f{c_1c_2 \mu_1}{(c_1 + c_2)^3} s}.
\end{equation}
Note that the right hand side (r.h.s.) in \eqref{Chap4:HTMarginal} corresponds to the LST of
an exponentially distributed random variable with mean $\f{c_1c_2 \mu_1}{(c_1 + c_2)^3}$.
\end{proof} 

\section{Joint workload analysis}\label{Chap4:Jointworkload}
We now focus on the joint workload distribution, restricting ourselves to the symmetric case, i.e.,  $c_1 = c_2 = c, \l_1 = \l_2 = \l$ and $\m_1 = \m_2 = \mu$. 
A main stepping stone in our analysis is the functional equation in~\eqref{EqJointstep15} below.
A corresponding functional equation can be derived for the asymmetric case, see also~\cite{Saxena2017}, but for the purpose of this paper it suffices to show that the symmetric case leads to a complicated BVP that, although it can be solved, provides little probabilistic insight to the problem at hand. 
Our next step is to analyze it under heavy-traffic, and, as explained earlier, the symmetric case is then the interesting one.

As a side remark, note that for both queues to reach heavy traffic simultaneously, it suffices to have $\lambda_1/\mu_1=\lambda_2/\mu_2$ if $c_1=c_2$; additionally demanding that $\lambda_1=\lambda_2$ (and hence $\mu_1=\mu_2$) amounts to choosing a different scaling unit for the workloads.\\

\textbf{Step 1:} Calculation of $\E\left(e  ^{-s_1 V_1(T_1) - s_2  V_2(T_1)}|V_1(0)=u_0,V_2(0)= l_0\right)$.\\
In this step, we calculate the LST of the joint workload distribution at time $T_1$. From the observations listed above the stability condition~\eqref{Chap4:Stability} in Section~\ref{Chap4:ModelDescription}, we obtain
\begin{align}
&\E\left(e^{-s_1   V_1(T_1) - s_2  V_2(T_1)}|V_1(0)=u_0,V_2(0)=l_0\right)
\nonumber\\
&= c e^{-s_2 l_0}  \left(\int_{t = 0}^{\frac{u_0}{\m - \l}} e^{- ( s_2  \l + c )t} e^{-s_1  (u_0 + (\l - \m)t)} {\rm d}t + \int_{t=\frac{u_0}{\m - \l}}^{\infty} e^{- ( s_2 \l + c)t} {\rm d}t\right) \nonumber \\
&= \frac{c}{ c +  s_2 \l} \left[\frac{s_2 \l + c}{s_2 \l + c + s_1(\l - \m)} e^{-s_1 u_0 - s_2 l_0} + \frac{ s_1(\l - \m)}{s_2 \l + c + s_1(\l - \m)} e^{- \frac{c + s_2 \l}{\m - \l}u_0 -s_2 l_0}\right] \label{Jointstep1}.
\end{align}

\textbf{Step 2:} Calculation of $\E\left(e^{-s_1 V_1(T_1) - s_2 V_2(T_1)}\right)
$ in stationarity:

In stationarity,  $\left(V_1(0), V_2(0)\right)$ and $\left(V_1(T_1+T_2), V_2(T_1+T_2)\right)$ have the same distribution.
Unconditioning, we obtain from \eqref{Jointstep1}:
\begin{align}
\label{EqJointWrk}
&\E\left(e^{-s_1 V_1(T_1) -s_2 V_2(T_1)}\right) \nonumber \\ &=\frac{c}{ c +  s_2 \l} \Big[\frac{s_2 \l + c}{s_2 \l + c + s_1(\l - \m)} \E\left(e^{-s_1 V_1(0) - s_2 V_2(0)}\right) \nonumber \\  
&\quad + \frac{ s_1(\l - \m)}{s_2 \l + c 
+ s_1(\l - \m)} \E\left(e^{- \frac{c + s_2 \l}{\m - \l} V_1(0) - s_2 V_2(0)}\right) \Big].
\end{align}

\paragraph{Formulation of a functional equation.} Since we are interested in the symmetric case,  we can formulate a functional equation corresponding to \eqref{EqJointWrk} by defining 
\begin{equation*}
\tilde{\nu}(s_1, s_2) := \E\left(e^{-s_1  V_1(0) -s_2  V_2(0)}\right),
\end{equation*} 
and then by symmetry,
\begin{equation*}
\tilde{\nu}(s_2, s_1) = \E\left(e^{-s_1 V_1(T_1) -s_2  V_2(T_1)}\right).
\end{equation*} 
Further, defining
\begin{equation}\label{function_f}
f(s_1, s_2)  := s_1 \l + c+  s_2(\l - \m),
\end{equation}
we obtain
\begin{align}
\label{EqJointstep1}
&\tilde{\nu}(s_2, s_1) = \frac{c}{f(s_2, s_1) } \tilde{\nu}(s_1, s_2) +  \frac{c}{ f(s_2, s_1)} \frac{ s_1(\l - \m)}{s_2 \l + c} \tilde{\nu}\left(\frac{s_2 \l + c}{\m - \l}, s_2\right).
\end{align}
Now  substituting $s_1 = s_2$, gives
\begin{align}
\label{EqJointstep11}
\tilde{\nu}\left(\frac{s_2 \l + c}{\m - \l}, s_2\right) =  \frac{(2 \l - \m)(s_2 \l +  c)}{c (\l - \m)} \tilde{\nu}(s_2, s_2).
\end{align}
Combining \eqref{EqJointstep1} and \eqref{EqJointstep11} yields
\begin{align}
\label{EqJointstep12}
\tilde{\nu}(s_2, s_1) = \frac{c}{f(s_2, s_1) } \tilde{\nu}(s_1, s_2) + \frac{(2\l - \mu)s_1}{f(s_2, s_1)} \tilde{\nu}(s_2, s_2).
\end{align}
By symmetry (after interchanging the indexes), 
\begin{align}
\label{EqJointstep13}
\tilde{\nu}(s_1, s_2) = \frac{c}{f(s_1, s_2) } \tilde{\nu}(s_2, s_1) +  \frac{(2\l - \mu)s_2}{f(s_1, s_2)} \tilde{\nu}(s_1, s_1).
\end{align}
Combining \eqref{EqJointstep12} and \eqref{EqJointstep13}, it follows that
\begin{align}
\label{EqJointstep14}
\tilde{\nu}(s_1, s_2) &= \frac{c^2}{f(s_1, s_2)f(s_2, s_1) } \tilde{\nu}(s_1, s_2) +    \frac{c(2\l - \m)s_1}{ f(s_1, s_2) f(s_2, s_1)} \tilde{\nu}(s_2, s_2) \nonumber \\
&\quad + \frac{(2\l - \mu)s_2}{f(s_1, s_2)} \tilde{\nu}(s_1, s_1),
\end{align}
so that finally, 
\begin{align}
\label{EqJointstep15}
&\f{\tilde{k}(s_1, s_2)}{s_1s_2} \tilde{\nu}(s_1, s_2)  = -2 \mu \left(\f{1}2 - \f{\l}{\mu}\right)\left[\f{c}{s_2}\ \tilde{\nu}(s_2, s_2) +   \f{f(s_1, s_2)}{s_1} \  \tilde{\nu}(s_1, s_1)\right],
\end{align}
with $\tilde{k}(s_1, s_2) = f(s_1, s_2) f(s_2, s_1) - c^2$ and with $f(s_1, s_2)$ defined in Equation \eqref{function_f}. 


Equations of this type have been studied in the monograph \cite{cohen2000boundary}. There a solution procedure for the present problem is outlined, which amounts to the following global steps:
\begin{description}
\item[Step A.] Consider the zeros of the {\em kernel} equation $\tilde{k}(s_1, s_2)$,
that have $\mathrm{Re}[s_1]$, $\mathrm{Re}[s_2] \geq 0$.
For such pairs $(s_1, s_2)$, $\tilde{\nu}(s_1, s_2)$ is analytic, and hence, for those pairs, the l.h.s.\ of \eqref{EqJointstep15} is equal to zero.

\item[Step B.] A suitable set of those zeros  of the kernel may form a contour. The fact that the r.h.s. of \eqref{EqJointstep15} is zero on that contour (the "boundary"), in combination with analyticity properties of $\tilde{\nu}(s_1,s_1)$ and $\tilde{\nu}(s_2,s_2)$ inside and/or outside that contour, can be used to formulate a Riemann or Riemann-Hilbert BVP. The solution of such a problem yields $\tilde{\nu}(s_1,s_1)$ and $\tilde{\nu}(s_2,s_2)$. Then $\tilde{\nu}(s_1,s_2)$ follows via \eqref{EqJointstep15}.
\end{description}

Unfortunately, the above steps do not constitute a simple, straightforward recipe.
For example, several choices of zero pairs are possible in the present problem,
and it is not a priori clear what is the best choice. Therefore, to successfully employ this Boundary Value method (BVM) requires a detailed investigation of the zeros of the kernel $\tilde{k}(s_1, s_2)$ of the functional equation. In what follows in this section, we describe in more detail these steps and emphasize the hurdles we encounter.

\paragraph{Kernel analysis.}
In the analysis of \eqref{EqJointstep15}, a crucial role is played by the kernel equation $\tilde{k}(s_1,s_2) = 0$.
Finding a suitable contour as mentioned above requires analyzing all pairs $(s_1,s_2)$ that solve the kernel equation, which is equivalent to
\begin{align}\label{kernel_s}
\lambda  (\lambda-\mu) s_2^2 + \Big[c(2 \lambda - \mu) &+ ((\lambda - \mu)^2 + \lambda^2)s_1\Big] s_2 \nonumber \\ &+   \lambda  (\lambda - \mu) s_1^2 + c (2\lambda - \mu) s_1 = 0,
\end{align}
with $\mathrm{Re}[s_1], \mathrm{Re}[s_2] \ge0$. By solving the above equation, we obtain the zeros of the kernel as
\begin{equation}\label{Chap4_ernel_root}
{s}_2^{\pm}(s_1) = \frac{- c(2 \lambda - \mu) - ((\lambda - \mu)^2 + \lambda^2)s_1  \pm (\mu-2 \lambda ) \mu   \sqrt{ \Delta(s_1) }}{2 \lambda  (\lambda -\mu )},
\end{equation}
with discriminant $\Delta(s_1) = s_1^2 - \f{c }{\mu } \f{1}{1/2 - \lambda /\mu } s_1  + \f{c^2}{\mu^2}$. The function ${s}_2^{\pm}(s_1)$ has two real branching points 
$$s_1^{\pm} = \f{c }{\mu } \f{1}{\nicefrac{1}{2} - \nicefrac{\lambda }{\mu} } \left(1 \pm \sqrt{1 - 4 \left(\nicefrac{1}{2} - \nicefrac{\lambda }{\mu}\right)^2}\right),$$ with $0 < s_1^{-} <  s_1^{+}$. Note that for $s_1\in (s_1^{-} ,  s_1^{+})$, ${s}_2^{\pm}(s_1) $ is a complex number, say ${s}_2^{\pm} =u+iv$ (where in the last equality we have suppressed the dependence on $s_1$). Noting that ${s}_2^{+}+{s}_2^{-} =2u$ and that ${s}_2^{+}\times  {s}_2^{-} =u^2+v^2$, we can define the contour that supports ${s}_2^{\pm}(s_1) $ for  $s_1\in (s_1^{-} ,  s_1^{+})$.
After cumbersome but straightforward computations, we obtain that 
\begin{align}
s_1&=\frac{c(\mu-2\lambda)+2\lambda u(\mu-\lambda)}{2 \lambda ^2 +\mu( \mu -2 \lambda )},\label{Eq:EllipseContour_s1}\\
r^2&=v^2+\frac{\left(u \kappa-\tau\right)^2}{\xi^2},
\label{Eq:EllipseContour}
\end{align}
with 
\begin{align*}
\kappa&=\sqrt{\left(3 \lambda ^2-3 \lambda  \mu +\mu ^2\right) \left(\lambda ^2-\lambda  \mu +\mu ^2\right)}, \\
\tau&=-c \mu  (\mu -2 \lambda)<0,\\
\xi&=2 \lambda ^2-2 \lambda  \mu +\mu ^2=2 \lambda ^2+\mu(\mu-2 \lambda)>0, \\
r^2&=\frac{c^2 \left(\lambda ^2-\lambda  \mu +\mu ^2\right) \left(\lambda  (\mu -\lambda ) \left(3 \lambda ^2-3 \lambda  \mu +\mu ^2\right)+(2 \lambda -\mu )^2\right)}{\lambda  (\mu -\lambda ) \left(2 \lambda ^2-2 \lambda  \mu +\mu ^2\right)^2}>0,
\end{align*}
which describes an ellipse for $0<\lambda<\nicefrac{\mu}{2}$. Let us denote the set by $$\tilde{E}  =  \set{(u, v) \in \left[\frac{\tau-r\xi}{\kappa},\frac{\tau +r\xi}{\kappa}\right]\times\mathbb{R} \, \Big| \,v^2+\frac{(u \kappa-\tau )^2}{\xi^2}=r^2 }.$$

\paragraph{BVM: Solution of the functional equation \eqref{EqJointstep15}.} Note that in order to solve the  functional equation, it suffices to compute $ \tilde{\nu}(s, s)$, for $\mathrm{Re}[s]\geq 0$. 

To this purpose, we take $s_1$ with $s_1\in (s_1^{-} ,  s_1^{+})$ and $s_2^{\pm}(s_1) = u \pm iv$, with $(u, v) \in \tilde{E}$. For all such $(s_1, s_2^{\pm}(s_1))$ pairs, the l.h.s. of \eqref{EqJointstep15} becomes zero, and hence, for all $s_2 = s_2^{\pm}(s_1)$, we have
\begin{align*}
\frac{\ \tilde{\nu}(s_1, s_1)}{s_1} &=\frac{-c}{f(s_1, s_2)}\frac{ \tilde{\nu}(s_2, s_2)}{ s_2}=\frac{f(s_2, s_1)}{-c s_2} \tilde{\nu}(s_2, s_2),
\end{align*}
where in the last equality we have used the fact that $(s_1,s_2)$ are roots of $\tilde{k}(s_1,s_2)=0$.
For $s_1 \in (s_1^{-} ,  s_1^{+})$, $\nicefrac{ \tilde{\nu}(s_1, s_1)}{ s_1}$ is real valued, thus, 
\begin{align*}
\mathrm{Re}\left[ -i f(s_2, s_1)  \tilde{\nu}(s_2, s_2)/ c s_2\right]&=0,
\end{align*}
with $f(s_2, s_1) / c s_2= (s_2 \l + c+  s_1(\l - \m)) / c s_2$.
For $s_2 = u + iv,\ (u, v) \in \tilde{E}$ and $s_1$ given in Equation  \eqref{Eq:EllipseContour_s1}, the above simplifies to
\begin{align}
-i\frac{ f(s_2, s_1)  }{c s_2 }&=
\frac{\lambda  v \left(2 u (\lambda -\mu )^2-c \mu \right)}{c\left(2 \lambda ^2-2 \lambda  \mu +\mu ^2\right) \left(u^2+v^2\right)} \nonumber\\
&\quad -i 
\frac{\lambda  \left(\mu  u (c+2 \lambda  u-\mu  u)+v^2 \left(2 \lambda ^2-2 \lambda  \mu +\mu ^2\right)\right)}{c \left(2 \lambda ^2-2 \lambda  \mu +\mu ^2\right) \left(u^2+v^2\right)}
\label{Eq:Gz_BVP_Def}\\
&:=a(u,v)+ib(u,v),\ (u,v)\in \tilde{E}.\label{Eq:Gz_BVP_Def_ab}
\end{align}
Next, we transform the problem into a Riemann-Hilbert problem on the unit circle $D$. For this purpose, we define $\tilde\phi$ (with inverse $\tilde\psi$) to be a conformal mapping of the interior of the unit circle $D$ onto the region bounded by $\tilde E$ with normalization conditions $\tilde\phi \left(-1\right)=\frac{\tau-r\xi}{\kappa}$, $ \tilde\phi (0)=\frac{2\tau}{\kappa}$, and  $\tilde\phi \left(1\right)=\frac{\tau+r\xi}{\kappa}$. 
That allows us to translate the Riemann-Hilbert BVP on and inside $\tilde E$ to the following Riemann-Hilbert BVP (cf.\ \cite[Section I.3.5]{cohen2000boundary} and \cite[Section 6]{boxma1993compensation}):
Let $D$ denote the unit circle contour and $D^+$ the interior of the unit circle,  then the BVP, with $a(\cdot)$ and $b(\cdot)$ real known functions defined on $D$,
\begin{equation}\label{BVPJointHT}
\mathrm{Re} [ \left(a(t)-ib(t)\right) h(t)] = 0, \, \, t\in  D,
\end{equation}
for some function $ h (\cdot)$ analytic in $D^+$ and continuous in $D^+\cup D$, has the following solution, cf.\ \cite{boxma1993compensation} and \cite[Section I.3.5]{cohen2000boundary},
\begin{equation}
 h (w) =  h _0\, \mathrm{Exp}\left(
\frac{1}{2\pi}
 \int_{t\in D}
 \mathrm{arctan}\left(
 \frac{b(t)}{a(t)}
 \right)
 \frac{t+w}{t-w}\frac{1}{t}\mathrm{d}t
 \right)
 , \, \, w \in D^+,
\end{equation}
where $ h _0$ is a constant and 
$$\mathrm{arctan}\left(
 \frac{b(t)}{a(t)}
 \right)
 =\frac{1}{2i}\log \left(\frac{a(t)+ib(t)}{a(t)-ib(t)}\right)
 .$$
Considering the conformal mapping from the ellipse $\tilde E$ to the unit circle $D$, say $\tilde\psi (\cdot)$, which is explicitly expressed in the Jacobi elliptic function (the sine of the amplitude  -- sinus amplitudinis -- or $\mathrm{sn}$, see, e.g., \cite[Sections 24--25]{Akhiezer}), yields, for $s_2$ inside the ellipse $\tilde{E}$,
\begin{equation}\label{Eq:Solution_BVP_original}
\tilde{\nu}(s_2, s_2) = 
  h _0\, \mathrm{Exp}\left(\frac{1}{4\pi i}
 \int_{t\in\tilde{E}}
\frac{1}{2i}\log \left(\frac{a(\tilde{\psi}(t))+ib(\tilde{\psi}(t))}{a(\tilde{\psi}(t))-ib(\tilde{\psi}(t))}\right)
 \frac{\tilde{\psi}(t)+s_2}{\tilde{\psi}(t)-s_2}\frac{1}{\tilde{\psi}(t)}\mathrm{d}\tilde{\psi}(t)
 \right),
\end{equation}
with $a(\cdot)$ and $b(\cdot)$ defined in \eqref{Eq:Gz_BVP_Def_ab}. 
The constant  $ h _0$ is determined from the normalizing condition $\tilde{\nu}(0,0)=1$.
With the above analysis,  we can compute the LST of the total workload, based on the conformal mapping $ \tilde{\psi}(\cdot)$. That enables us to explicitly determine $\tilde{\nu}(s_1, s_2)$ as defined in Equation \eqref{EqJointstep15}. As evident from Equation \eqref{Eq:Solution_BVP_original}, this is quite cumbersome and typically leads to expressions in which one needs to perform a difficult computational procedure as they involve inverting the LST, which is in terms expressed using the Jacobi elliptic function. In addition to the numerical complications, due to the nature of the solution of the BVP, it is difficult to gain probabilistic insight into the problem at hand. 

In addition to the above mentioned hurdles, it is also important to note that  by definition $\tilde{\nu}(s_2, s_2)$ is analytic for $\mathrm{Re}[s_2]=u \geq 0$, but the domain $\tilde{E}$ requires the analytic continuation of $\tilde{\nu}(s_2, s_2)$ to $\mathrm{Re}[s_2]=u \geq \nicefrac{(\tau-r\xi)}{\kappa}$ (note that $\nicefrac{(\tau-r\xi)}{\kappa}<0$). This would constitute one further hurdle in the analysis. 

For all aforementioned reasons, we instead focus on the heavy-traffic setting of the model and solve the resulting simpler BVP.

Note that the above analysis is very similar to the one performed in \cite[Section III.1]{cohen2000boundary}, as also there the problem at hand (of two queues in parallel under the join the shortest queue routing protocol) yields a Riemann-Hilbert problem on an ellipse, cf. \cite{OnnoJacquesIvo}. Because of the similarities between the two problems, one could further investigate other possible equivalent expressions to \eqref{Eq:Solution_BVP_original} pertaining to a meromorphic expansion of the equation which could be explicitly inverted, cf. \cite[Section III.1.4, Equation (4.11)]{cohen2000boundary}.

\section{Heavy-traffic analysis of the joint workload distribution}\label{Chap4:HT_Analysis}
In this section, we shall determine the heavy-traffic limit of the LST of the scaled joint workload distribution of the symmetric model in stationarity. 
In what follows, we use functional equation \eqref{EqJointstep15}. Let $\rho$ be the load on each of the two queues  ($\rho = \nicefrac{\l}{\mu}$).
We scale the functional equation by replacing $s_1$ by $\left(\nicefrac{1}{2} - \rho\right)s_1$, and $s_2$ by $\left(\nicefrac{1}{2} - \rho\right)s_2$. After dividing by $-2 \mu c$ in \eqref{EqJointstep15} and taking the limit $\rho \uparrow \nicefrac{1}{2}$, we obtain the following functional equation
\begin{align}
\label{Scaling2HT1}
\f{\tilde{k}^{\star}(s_1, s_2)}{s_1 s_2} \tilde{\nu}^{\star}(s_1, s_2)  &=  \f{\tilde{\nu}^{\star}(s_1, s_1)}{s_1} + \f{\tilde{\nu}^{\star}(s_2, s_2)}{s_2},
\end{align}
where $\tilde{\nu}^{\star}(s_1, s_2) = \lim\limits_{\rho \uparrow \nicefrac{1}{2}}  \E(e^{-s_1\left(\nicefrac{1}{2} - \rho\right)V_1 -s_2\left(\nicefrac{1}{2} - \rho\right)V_2 })$ and
\begin{align}\label{Kernel}
\tilde{k}^{\star}(s_1, s_2)
&= -\lim\limits_{\rho \uparrow \nicefrac{1}{2}} \f{1}{2 \mu c\left(\nicefrac{1}{2} - \rho\right)^2}~~ \tilde{k}\left(\left(\nicefrac{1}{2} - \rho\right)s_1,\left(\nicefrac{1}{2} - \rho\right)s_2\right) \nonumber \\
&= s_1   + s_2 + \frac{\mu}{8c}\big(s_1 - s_2\big)^2 .
\end{align}

There is one unknown function in the r.h.s.\ of \eqref{Scaling2HT1}: $\tilde{\nu}^{\star}(s, s)$.
We calculate this unknown function using the BVM by applying Step A and Step B discussed in Section \ref{Chap4:Jointworkload}.

\paragraph{Kernel analysis.}
To apply the BVM, one needs to investigate the zeros of kernel $\tilde{k}^{\star}(s_1, s_2)$. By setting $\tilde{k}^{\star}(s_1, s_2) = 0$, we obtain
\begin{align}
s_2^{\pm}(s_1)
&= \f{-1 +  \f{\mu}{4c} s_1 \pm \sqrt{1 - \frac{\mu}{c} s_1}}{\f{\mu}{4c}}.
\end{align}
Note that $s_2^{\pm}(s_1)$ has a single branching point at $s_1 = \f{c}{\mu}$. For real valued $s_1$ with $s_1 > \nicefrac{c}{\mu}$, the function $s_2^{\pm}(s_1)$ is complex valued. Letting $s_2^{\pm}(s_1) = u \pm i v$, we obtain
\begin{align}\label{KernelEquation}
u^2  + v^2 = s_2^{+}(s_1) s_2^{-}(s_1) 
&= \f{\left(-1 +  \f{\mu}{4c} s_1\right)^2 - 1 + \frac{\mu}{c} s_1}{\left( \f{\mu}{4c} \right)^2},
\end{align}
and
\begin{align}
2 u = s_2^{+}(s_1) + s_2^{-}(s_1) 
= \f{-1 +  \f{\mu}{4c} s_1}{\f{\mu}{8c}}.
\end{align}
Computing $s_1 = u + \nicefrac{3c}{\mu}$ from the above equation and substituting it into \eqref{KernelEquation}, we have
\begin{align}\label{KernelEquation1}
u^2  + v^2 &= \f{\left(\f{\mu}{4c}\right)^2 u^2 - 1 + \f{\mu}{c}u + 4}{\left(\f{\mu}{4c}\right)^2}.
\end{align}
Simplifying the above equation yields
\begin{align}\label{Parabola}
v^2 
&= \f{16 c}{\mu} \left(u + \f{3 c}{\mu}\right),
\end{align}
which describes a parabola in the complex plane. We will restrict ourselves to the following set:
$$E  =  \set{(u, v) \in (-\f{3c}{\mu}, \infty) \times \mathbb{R} \mid v^2 = \f{16 c}{\mu} \left(u + \f{3 c}{\mu}\right) }.$$

\paragraph{BVM: Solution of the functional equation \eqref{Scaling2HT1}.}
Now we take $s_1$ with $s_1 > \nicefrac{c}{\mu}$ and $s_2^{\pm}(s_1) = u \pm iv$, with $(u, v) \in E$. For all such $(s_1, s_2^{\pm}(s_1))$ pairs, the r.h.s. of \eqref{Scaling2HT1} becomes zero, and hence, for all $s_2 = s_2^{\pm}(s_1)$, we have
\begin{equation}
\f{\tilde{\nu}^{\star}(s_2, s_2)}{s_2} = - \f{\tilde{\nu}^{\star}(s_1, s_1)}{s_1}.
\end{equation}
For $s_1 > \nicefrac{c}{\mu}$, the r.h.s. of the above equation is real, thus yielding
\begin{equation}
\mathrm{Re}\left[g(s_2) \right] = 0, \, \, \mbox{for} \, \, s_2 =  s_2^{\pm}(s_1) = u \pm iv, ~\mbox{with}~ (u, v) \in E \backslash \{(0, 0)\},
\end{equation}
where $g(s_2) = i \f{\tilde{\nu}^{\star}(s_2, s_2)}{s_2}$.
Notice that  $\tilde{\nu}^{\star}(s_2, s_2)$ is analytic for $\mathrm{Re}[s_2] \geq 0$.
Below, we  prove, in Lemma \ref{Chap4:Claim}, that $\tilde{\nu}^{\star}(s_2, s_2)$ is analytic on the strip $- \nicefrac{3 c}{\mu} < \mathrm{Re}[s_2] < 0$. 
For clarity of exposition, we postpone the proof of this lemma until after Theorem~\ref{HT_LST_Total_workload}, at which point we will have introduced all necessary notation.

We thus see that $g(s_2)$ is analytic inside the contour $E$, say $E^+$, except for $s_2=0$ which is a pole in $E^+$. The above problem now reduces to a Riemann-Hilbert problem with a pole, and with boundary $E$,  see \cite[Section I.3.3]{cohen2000boundary}.  
To transform it into a (standard) Riemann-Hilbert problem on the unit circle $D$,
we define $\phi$ (with inverse $\psi$) to be a conformal mapping of the interior of the unit circle $D$ onto the region bounded by $E$ with normalization conditions $\phi\left(-1\right)=\infty$, $\phi(0)=0$, and  $\phi\left(1\right)=-\nicefrac{3 c}{\mu}$. 
That allows us to translate the Riemann-Hilbert BVP on and inside $E$ to 
the following simple Riemann-Hilbert BVP with a pole (cf.\ \cite[Section I.3.3]{cohen2000boundary} and \cite[Section 6]{boxma1993compensation}):
Defining $h(w) := g(\phi(w))$, we obtain 
for $h(\cdot)$ on the unit circle $D$:
\begin{equation}\label{BVPJointHT}
\mathrm{Re} [h(w)] = 0, \, \, w \in  D \backslash \{0\},
\end{equation}
with $h(\cdot)$ analytic in $D^+\backslash \{0\}$ and continuous in $D^+\cup D\backslash \{0\}$.   The solution of the BVP \eqref{BVPJointHT} is, cf.\ \cite{boxma1993compensation} and \cite[Section I.3.3]{cohen2000boundary}:
\begin{equation}
h(w) = i  \alpha + \beta  w - \f{\bar{\beta } }{w}, \, \, w \in D^+\cup D\backslash \{0\},
\end{equation}
where $\alpha , \beta $ are constants that we will calculate explicitly in Theorem \ref{HT_LST_Total_workload}.
This determines $g(x) = h(\psi(x))$; substituting it in the above equation we obtain  
\begin{equation}
g(s_2) = i  \alpha + \beta  \psi(s_2) - \f{\bar{\beta } }{\psi(s_2)}, \, \, s_2 \in E^+\cup E\backslash\{0\},
\end{equation}
where $\psi(\cdot)$ is the conformal mapping from the parabola $E$ to the unit circle $D$. Since $g(s_2) = i {\tilde{\nu}^{\star}(s_2, s_2)}/{s_2}$, we obtain for $\mathrm{Re}[s_2] > - \nicefrac{3 c}{\mu}$,
\begin{equation}\label{v^{star}(s_2, s_2)}
\tilde{\nu}^{\star}(s_2, s_2) =   \alpha s_2 - i~\beta ~  \psi(s_2) s_2 + i \f{\bar{\beta } s_2}{\psi(s_2)}.
\end{equation}
With that we have calculated the LST of the total workload in heavy traffic, based on the conformal mapping $ \psi(s_2)$. Before materializing this in Theorem~\ref{HT_LST_Total_workload}, we give an explicit expression for $\psi(s_2)$ in the next lemma.
That will enable us to explicitly determine $\tilde{\nu}^{\star}(s_2, s_2)$ in Theorem~\ref{HT_LST_Total_workload}.

\begin{lemma}\label{lemma:ConformalMap}
For $z  \in \mathbb{C}$, we have a conformal map $\psi(z)$ which maps the interior of parabola \eqref{Parabola} onto the interior of the unit circle $D$, and it is given explicitly as follows:
\begin{equation}\label{ConformalMap}
\psi(z) = \f{1 - \sqrt{2}\cosh(\f{\pi}{4}\sqrt{\f{\mu}{c} z - 1})}{1 + \sqrt{2} \cosh(\f{\pi}{4}\sqrt{\f{\mu}{c} z - 1})}.
\end{equation}
\end{lemma}
\begin{proof}
The conformal mapping $\psi(z)$ is obtained by taking the composition of the following conformal mappings:
\begin{itemize}
\item[i.] The conformal mapping $\eta(z) = z - \f{c}{\mu}$, where $z = x + i y$, maps parabola $y^2 = \f{16c}{\mu} (x + \f{3c}{\mu})$ onto parabola $y^2 =  \f{16c}{\mu} (x + \f{4c}{\mu})$.
\item[ii.] From \cite[p.113]{bieberbach2000conformal}, we have that the conformal mapping $\xi(z) =  i  \cosh\left(\f{\pi}{4} \sqrt{\f{\mu}{c}z}\right)$ maps the interior of the parabola $y^2 =  \f{16c}{\mu} (x + \f{4c}{\mu})$ onto the interior of the upper half-plane  $\mathrm{Im}[\xi] > 0$.  
\item[iii.] As shown in \cite[p. 326, Equation (6)]{brown2009complex},  the conformal mapping $w(z) = \f{1 + i \sqrt{2}  z }{1 - i \sqrt{2} z}$ maps the upper half-plane (i.e., $\mathrm{Im}[z] > 0$ ) onto the interior of the unit circle $|w| <1$.
\end{itemize}

It follows from \cite[Theorem III]{bieberbach2000conformal}, that a composition of conformal mappings again is a conformal mapping. Hence the composition mapping $\psi(z)  := w(\xi(\eta(z)))$  conformally  maps the interior of the parabola \eqref{Parabola} onto the interior of the unit circle $D$. 
\end{proof}

Now we state the first main theorem of this section, in which we obtain an explicit expression  for the total stationary workload LST in heavy traffic. 

\begin{theorem}\label{HT_LST_Total_workload}
The scaled total workload LST in heavy traffic is given by, for $\mathrm{Re}[s] > -\nicefrac{3 c}{\mu}$,
\begin{equation}\label{Chap4:LST_Total_worload}
\lim\limits_{\rho \uparrow \nicefrac{1}{2}}  \E\left( e^{-s \left(\nicefrac{1}{2} - \rho\right)(V_1 +V_2) }\right ) = \f{\pi}{4}\f{\mu}{c} \f{s}{\cosh(\f{\pi}{2}\sqrt{\f{\mu}{c} s - 1})}.
\end{equation}
\end{theorem}
\begin{proof}
Substituting $\psi(z)$ from Lemma \ref{lemma:ConformalMap} in \eqref{v^{star}(s_2, s_2)} yields 
\begin{align}
\tilde{\nu}^{\star}(s, s) &= \alpha s -  i ~\beta ~\left(\f{1 - \sqrt{2}\cosh(\f{\pi}{4}\sqrt{\f{\mu}{c} s - 1})}{1 + \sqrt{2} \cosh(\f{\pi}{4}\sqrt{\f{\mu}{c} s - 1})}\right)s \nonumber \\ 
&\quad + i ~\bar{\beta }~\left(\f{1 + \sqrt{2}\cosh(\f{\pi}{4}\sqrt{\f{\mu}{c} s - 1})}{1 - \sqrt{2} \cosh(\f{\pi}{4}\sqrt{\f{\mu}{c} s - 1})}\right)s.
\end{align}
Since $\tilde{\nu}^{\star}(0, 0) = 1$, we obtain from the above equation $\bar{\beta } = \f{\pi}{16}\f{\mu}{c} i$, and since $\tilde{\nu}^{\star}(\infty, \infty) = 0$, we obtain $\alpha = -\f{\pi}{8}\f{\mu}{c}$. Substituting the values of  $\alpha $, $\beta $ and $\bar{\beta }$ into the above equation and thereafter simplifying it, we obtain 
\begin{align}
&\tilde{\nu}^{\star}(s, s) \nonumber \\ &= -\f{\pi}{8}\f{\mu}{c} \left[1 + \f{1}{2} \f{1 - \sqrt{2}\cosh(\f{\pi}{4}\sqrt{\f{\mu}{c} s - 1})}{1 + \sqrt{2} \cosh(\f{\pi}{4}\sqrt{\f{\mu}{c} s - 1})} + \f{1}{2} \f{1 + \sqrt{2}\cosh(\f{\pi}{4}\sqrt{\f{\mu}{c} s - 1})}{1 - \sqrt{2} \cosh(\f{\pi}{4}\sqrt{\f{\mu}{c} s - 1})}\right] s.
\label{Eq:38_IntermediateResultNU}
\end{align}
The theorem follows after some further simplifications and using  the trigonometric square formula $\cosh^2 x=\left(\cosh (2x)+1\right)/2$.
\end{proof}

It is now convenient to formulate and prove the postponed Lemma~\ref{Chap4:Claim}.
As we have discussed in Step A of Section \ref{Chap4:Jointworkload}, we are interested in finding the LST in the domain $\mathrm{Re}[s_2] \geq 0$.  In the previous theorem, we have calculated the LST expression $\tilde{\nu}^{\star}(s_2, s_2) $ in $\mathrm{Re}[s_2] > -\nicefrac{3 c}{\mu}$. We want to show that $\tilde{\nu}^{\star}(s_2, s_2) $ is analytic in the strip $-\nicefrac{3 c}{\mu} < \mathrm{Re}[s_2] < 0$. From \eqref{Chap4:LST_Total_worload}, we have an explicit expression and it is sufficient to show that the denominator $\cosh(\f{\pi}{2}\sqrt{\f{\mu}{c} s_2 - 1})$ has no zeros on that strip. 
\begin{lemma}\label{Chap4:Claim}
The LST of the total scaled workload in heavy traffic, as given in \eqref{Chap4:LST_Total_worload}, is analytic on the strip $-\nicefrac{3 c}{\mu} < \mathrm{Re}[s_2] < 0$.
\end{lemma} 
\begin{proof}
In the proof of Lemma \ref{lemma:ConformalMap}, we observed that $\cosh(\f{\pi}{2}\sqrt{\f{\mu}{c} s_2 - 1})$ is a conformal mapping for $\mathrm{Re}[s_2] > -\nicefrac{3 c}{\mu}$, and hence it is an analytic function for $\mathrm{Re}[s_2] > -\nicefrac{3 c}{\mu}$. Moreover, the reciprocal of this analytic function is also analytic (see \cite[p. 74]{brown2009complex}) if the denominator has no zeros in that domain. To show that the denominator $\cosh(\f{\pi}{2}\sqrt{\f{\mu}{c} s_2 - 1})$ has no zeros in  $-\nicefrac{3 c}{\mu} < \mathrm{Re}[s_2] < 0$, we solve 
\begin{equation}
0 = \cosh(\f{\pi}{2}\sqrt{\f{\mu}{c} s_2 - 1}) = \f{e^{\f{\pi}{2}\sqrt{\f{\mu}{c} s_2 - 1}} + e^{-\f{\pi}{2}\sqrt{\f{\mu}{c} s_2 - 1}}}{2},
\end{equation}
so
\begin{equation}
e^{\f{\pi}{2}\sqrt{\f{\mu}{c} s_2 - 1}} =  e^{\pi i -\f{\pi}{2}\sqrt{\f{\mu}{c} s_2 - 1}},
\end{equation}
and hence  $\f{\pi}{2}\sqrt{\f{\mu}{c} s_2 - 1} = \pi i -\f{\pi}{2}\sqrt{\f{\mu}{c} s_2 - 1} + 2  \pi n i, ~ n \in \mathbb{Z}$. This implies that the zeros of the function $\cosh(\f{\pi}{2}\sqrt{\f{\mu}{c} s_2 - 1})$ are $s_2 = \f{c}{\mu}\left(1 - (2n + 1)^2\right), ~n \in \mathbb{Z}$. There are two different cases for the zeros we need to discuss: (i) when $n = 0$ or $n = -1$, we have $s_2  = 0$, and in this case we know $\tilde{\nu}^{\star}(s_2, s_2)$ is $1$. (ii) When $n \in \mathbb{Z} \backslash \{0, -1\}$, we have $s_2 = \f{c}{\mu} \left(1 - (2n + 1)^2\right) \leq - \nicefrac{8c}{\mu} < -\nicefrac{3 c}{\mu}$, which concludes the claim of the lemma. 

Note that instead of working directly with the roots appearing in the simplified Equation \eqref{ConformalMap}, one could consider the roots of the two denominators appearing in  Equation \eqref{Eq:38_IntermediateResultNU}, i.e., the zeros of $1\pm \sqrt{2}\cosh(\f{\pi}{4}\sqrt{\f{\mu}{c} s_2 - 1}) $.

Equivalently, one can prove the analytic continuation using  Euler's formula, cf. \cite[Equation (3.3)]{biane2001probability}, which converts the hyperbolic cosine into an infinite product (we use this approach to rewrite Equation \eqref{ConformalMap} as an infinite product expansion, cf. Equation \eqref{Eq:Prod_Expr_FPT_tildeC}). Using the infinite product expansion, it becomes evident that, in the domain $ \mathrm{Re}[s_2]>-\nicefrac{8 c}{\mu}$, there are no roots of the denominator.
\end{proof}

The LST of the total workload lends itself to explicitly determine the heavy-traffic stationary workload {\em distribution} as shown in the following Lemma:

\begin{lemma}\label{Chap4:Lemma:density_Total_workload}
With $f_{(\nicefrac{1}{2}- \rho)(V_1 +  V_2)}(\cdot)$ the probability density function of the scaled total workload $(\nicefrac{1}{2}- \rho)(V_1 +  V_2)$, we have
\begin{align}\label{Eq:Final_PDF_V1}
 \lim\limits_{\rho \uparrow \nicefrac{1}{2}}  f_{(\nicefrac{1}{2}- \rho)(V_1 +  V_2)}(x)  = 
\sum_{n=1}^\infty(-1)^{n+1}(2n+1)\frac{c}{\mu}\left((2n+1)^2-1\right) e^{-\frac{c}{\mu}\left((2n+1)^2-1\right)x},\ x>0.
\end{align}
Moreover, the limiting distribution as $\rho\uparrow 1/2$ of the scaled total workload $(\nicefrac{1}{2}- \rho)(V_1 +  V_2)$ is infinitely divisible and is distributed like
\begin{align}\label{Eq:ID_V1}
\sum_{n=1}^{\infty}\frac{\mathcal{E}_n}{\frac{c}{\mu}\left((2n+1)^2-1\right)},
\end{align}
where $\{\mathcal{E}_n\}_{n\in\mathbb{N}}$ is a sequence of independent and identically exponentially distributed random variables with rate 1. 
\end{lemma}

The infinite divisibility of the scaled total workload distribution is a consequence of the infinite divisibility of the exponential distribution. 

Before proceeding with the proof of Lemma \ref{Chap4:Lemma:density_Total_workload}, we review the needed relevant results in the remark below. 

\begin{remark}\label{Remark:ID_PDF_FPT}
To compute the limiting probability density function of the scaled total workload in heavy traffic, we need to invert the LST \eqref{Chap4:LST_Total_worload}. The appearance of LSTs with a hyperbolic cosine and their probabilistic interpretation has a long standing tradition in probability theory, see, e.g., \cite{biane2001probability} and the references therein. As we shall need these results for the proof of Lemma \ref{Chap4:Lemma:density_Total_workload}, we review them shortly below.

Consider a random variable defined as
\begin{align}\label{Eq:ID_FPT1}
C=\frac{2}{\pi^2}\sum_{n=1}^{\infty}\frac{\mathcal{E}_n}{(n-\nicefrac{1}{2})^2}
\end{align}
with $\{\mathcal{E}_n\}_{n\in\mathbb{N}}$ a sequence of independent and identically exponentially distributed random variables with rate 1. Then,
\begin{align}\label{Eq:DerivationsLST_jD_FPT}
\E\left(e^{-sC}\right)=\E\left(\prod_{n=1}^{\infty}e^{-\frac{2s}{\pi^2 (n-\nicefrac{1}{2})^2}\mathcal{E}_n}\right)=\frac{1}{\prod_{n=1}^{\infty}\left(1+\frac{2s}{\pi^2 (n-\nicefrac{1}{2})^2}\right)}=\frac{1}{\cosh \sqrt{2s}},
\end{align}
where the last equality is known as Euler's formula, cf. Equation (3.3) in \cite{biane2001probability}. Moreover, using the Mittag-Leffler expansion, based on the poles of the r.h.s. of Equation \eqref{Eq:DerivationsLST_jD_FPT}, yields
\begin{align*}
\E\left(e^{-sC}\right)=\pi\sum_{n=1}^\infty \frac{(-1)^n(n-\nicefrac{1}{2})}{s+(n-\nicefrac{1}{2})^2\pi^2/2},
\end{align*}
cf.  \cite[Equation (2.21)]{Taylor_1962}. Noting that 
$\frac{1}{s+(n-\nicefrac{1}{2})^2\pi^2/2}=\int_{x=0}^{\infty} e^{-sx}e^{- (n-\nicefrac{1}{2})^2\pi^2x/2}\mathrm{d}x$, this last expression yields the density function of the random variable $C$, more concretely
\begin{align}\label{Eq:pdf_C2}
f_{C}(x)=\pi \sum_{n=1}^\infty (-1)^n (n-\nicefrac{1}{2})e^{-(n-\nicefrac{1}{2})^2\pi^2 x/2},\ x>0.
\end{align}
Moreover, equivalent expressions to \eqref{Eq:pdf_C2} can be produced using the reciprocal relation $f_C(x)=\left(\frac{2}{\pi x}\right)^{3/2}f_C\left(\frac{4}{\pi^2 x}\right)$, cf. \cite[Table 1 (continued) Row 5]{biane2001probability}. This immediately implies that 
\begin{align*}
f_{C}(x)=\sqrt{\frac{2}{\pi x^3}}\sum_{n=1}^\infty (-1)^n (2n-1)e^{-(2n-1)^2/2x},\ x>0,
\end{align*}
see \cite[Equation (3.11)]{biane2001probability}. 

As stated in \cite[page 441]{Taylor_1962}, this turns out to be the density for the maximum displacement of a one-dimensional standard Brownian motion in a fixed time interval or, as stated in \cite[Table 2 Row 3]{biane2001probability}, the density of the hitting time of $1$ of the one-dimensional standard Brownian motion with reflection at $0$.

To further understand the infinite divisibility of the hitting time, the interested reader is referred to \cite[page 550]{Feller}, where the idea relies on the fact that the  hitting time from 0 to 1 can be divided into the hitting time from 0 to any point in the interval $(0,1)$ plus the independent (by the strong Markov property) hitting time from that point to $1$. By putting more and more points between $0$ and $1$, the hitting time can be expressed as the limit of a null triangular array, hence giving rise to the infinite divisibility property expressed  in \eqref{Eq:ID_FPT1}.

A similar approach can be applied for the random hitting time of a one-dimensional standard Brownian motion with drift $\mu\geq 0$ to $\{\pm 1\}$, say $C'$. As shown in \cite[Theorem 7.1]{Kent_1978}, the random hitting time has the following representation 
\begin{align}\label{Eq:ID_FPT2}
C'=2\sum_{n=1}^{\infty}\frac{\mathcal{E}_n}{\mu^2+\pi^2(n-\nicefrac{1}{2})^2},
\end{align}
and it is shown, by performing the same computations as in \eqref{Eq:DerivationsLST_jD_FPT}, to have the following LST
\begin{align*}
\E\left(e^{-sC'}\right)=\frac{\cosh \mu}{\cosh \sqrt{2s+\mu^2}}.
\end{align*}
\end{remark}

\begin{proof}[Proof of Lemma \ref{Chap4:Lemma:density_Total_workload}] We  express the LST  \eqref{Chap4:LST_Total_worload} as an infinite product of LST of independent exponentially distributed random variables. To this purpose, we need the following two identities
\begin{align}\label{Eq:ML_Expans}
\frac{1}{\cosh \sqrt{s}}=\prod_{n=1}^\infty \left(1+\frac{s}{\pi^2(n-\nicefrac{1}{2})^2}\right),
\end{align}
which is  Euler's formula, cf. \cite[Equation (3.3)]{biane2001probability}. 
Moreover, 
\begin{align*}
\cos \pi s =\prod_{n=0}^\infty  \left(1-\left(\frac{s}{n+\nicefrac{1}{2}}\right)^2\right).
\end{align*}
From this last equation, by taking out the $n=0$ term, we can show that
\begin{align}\label{Eq:Cos_Zero}
\prod_{n=1}^\infty  \left(1-\left(\frac{\nicefrac{1}{2}}{n+\nicefrac{1}{2}}\right)^2\right)=\lim_{s\to1/2}\frac{\cos \pi s}{1-4s^2}=\frac{\pi}{4}.
\end{align}
Using \eqref{Eq:ML_Expans} and \eqref{Eq:Cos_Zero}, yields after straightforward computations that
\begin{align}
\f{\pi}{4}\f{\mu}{c} \f{s}{\cosh(\f{\pi}{2}\sqrt{\f{\mu}{c} s - 1})} &= \frac{\mu s}{c}\frac{\prod_{n=1}^\infty  \left(1-\left(\frac{\nicefrac{1}{2}}{n+\nicefrac{1}{2}}\right)^2\right)}{\prod_{n=1}^\infty \left(1+\frac{
\frac{\pi^2 \mu s}{4c}-\frac{\pi^2}{4}}{\pi^2(n-\nicefrac{1}{2})^2}\right)}\nonumber\\
&= \frac{\prod_{n=2}^\infty  \left(1-\left(\frac{\nicefrac{1}{2}}{n-1/2}\right)^2\right)}{\prod_{n=2}^\infty \left(1+\frac{\frac{ \mu s}{4c}-\frac{1}{4}}{(n-\nicefrac{1}{2})^2}\right)}\nonumber\\
&=\prod_{n=1}^\infty\frac{\frac{c}{\mu}\left((2n+1)^2-1\right)}{s+\frac{c}{\mu}\left((2n+1)^2-1\right)}.\label{Eq:Prod_Expr_FPT_tildeC}
\end{align}
Note that the last equality reveals that the LST at hand is associated with the random variable of Equation \eqref{Eq:ID_V1}. For convenience, we shall denote the random variable of Equation \eqref{Eq:ID_V1} by $\tilde{C}$.

We now turn our attention to the computation of the density function. Note that the conventional approach to produce the density function (based on a meromorphic expansion) doesn't work as the corresponding (meromorphic) series diverges. We shall overcome this following the approach of \cite{Taylor_1962}. More concretely, we consider, for $N>0$, 
\begin{align}\label{Eq:Mer_Expa_Sum_FPT}
\prod_{n=1}^N\frac{\frac{c}{\mu}\left((2n+1)^2-1\right)}{s+\frac{c}{\mu}\left((2n+1)^2-1\right)}&=
\sum_{n=1}^N\frac{\frac{c}{\mu}\left((2n+1)^2-1\right)}{s+\frac{c}{\mu}\left((2n+1)^2-1\right)}
\prod\limits_{k=1,k\neq n}^N
\frac{k(k+1)}{k(k+1)-n(n+1)}\nonumber\\
&=\sum_{n=1}^N\frac{\frac{c}{\mu}\left((2n+1)^2-1\right)}{s+\frac{c}{\mu}\left((2n+1)^2-1\right)}
\frac{(-1)^{n+1} (2 n+1) N!  (N+1)!}{(N-n)!  (N+n+1)!}\nonumber\\
&=\sum_{n=1}^\infty(-1)^{n+1} (2 n+1)\frac{\frac{c}{\mu}\left((2n+1)^2-1\right)}{s+\frac{c}{\mu}\left((2n+1)^2-1\right)}\nonumber\\
&\quad\quad\quad\times\frac{ (N-n+1)\cdots N }{(N+2)\cdots (N+n+1)}\mathds{1}_{\{n\leq N\}}
,
\end{align}
by taking partial fractions and noting that $(2n+1)^2-1=4n(n+1)$. Note that, as $N\to\infty$, the l.h.s. of Equation \eqref{Eq:Mer_Expa_Sum_FPT} converges to \eqref{Eq:Prod_Expr_FPT_tildeC}.
The Laplace transform on the r.h.s. of Equation \eqref{Eq:Mer_Expa_Sum_FPT} can be easily inverted, from which we obtain that the density function of \eqref{Eq:Prod_Expr_FPT_tildeC} is given by 
\begin{align}\label{Eq:PDF_Sum_FPT}
\lim_{N\to\infty}
\sum_{n=1}^\infty & (-1)^{n+1} (2 n+1) \frac{c}{\mu}\left((2n+1)^2-1\right) e^{-\frac{c}{\mu}\left((2n+1)^2-1\right)x}\nonumber\\
&\times\frac{ (N-n+1)\cdots N }{(N+2)\cdots (N+n+1)}\mathds{1}_{\{n\leq N\}},\ x>0.
\end{align}
Applying the Dominated Convergence Theorem immediately yields Equation \eqref{Eq:Final_PDF_V1} as the terms (with respect to $N$) inside the series are bounded
\begin{align*}
\frac{ (N-n+1)\cdots N }{(N+2)\cdots (N+n+1)}\leq 1,\ \forall \ 1\leq n\leq N,
\end{align*}
as
\begin{align*}
\lim_{N\to\infty}&
(-1)^{n+1} (2 n+1) \frac{c}{\mu}\left((2n+1)^2-1\right) e^{-\frac{c}{\mu}\left((2n+1)^2-1\right)x}
\frac{ (N-n+1)\cdots N }{(N+2)\cdots (N+n+1)}\mathds{1}_{\{n\leq N\}}\\
&=(-1)^{n+1} (2 n+1) \frac{c}{\mu}\left((2n+1)^2-1\right) e^{-\frac{c}{\mu}\left((2n+1)^2-1\right)x},
\end{align*}
and the series 
\begin{align*}
\sum_{n=1}^\infty & (-1)^{n+1} (2 n+1) \frac{c}{\mu}\left((2n+1)^2-1\right) e^{-\frac{c}{\mu}\left((2n+1)^2-1\right)x}
\end{align*}
converges for $x>0$.
From  \cite[Theorem 8.2, page 60]{9}, it follows that \eqref{Eq:Final_PDF_V1} is indeed the density function in question. This is intuitively validated by noting that, if we were to directly use the Mittag-Leffler expansion, based on the poles of the r.h.s. of Equation \eqref{Eq:Prod_Expr_FPT_tildeC}  this would  yield
\begin{align*}
\f{\pi}{4}\f{\mu}{c} \f{s}{\cosh(\f{\pi}{2}\sqrt{\f{\mu}{c} s - 1})} &=\sum_{n=1}^\infty  (-1)^{n+1}(2n+1)\frac{\frac{c}{\mu}\left((2n+1)^2-1\right)}{s+\frac{c}{\mu}\left((2n+1)^2-1\right)}.
\end{align*}
However the r.h.s. of the above equation does not converge, but still yields the same result as in \eqref{Eq:Final_PDF_V1}. This was noticed and commented upon on \cite[page 441]{Taylor_1962}.
\end{proof}

We now state the most important result of this section.
In \eqref{Chap4:LST_Total_worload}, we have computed an explicit expression for the scaled {\em total} workload LST in heavy traffic. In the following theorem, we give an explicit expression for the scaled {\em joint} workload LST in heavy traffic.

\begin{theorem}\label{thm:htjoint}
For $\mathrm{Re}[s_j] > - \nicefrac{3c}{\mu}$, $j=1,2$, the scaled joint workload LST in heavy traffic is given by:
\begin{equation}\label{v^{star}(s_1, s_2)}
\lim\limits_{\rho \uparrow \nicefrac{1}{2}}  \E(e^{-s_1\left(\nicefrac{1}{2} - \rho\right)V_1 -s_2\left(\nicefrac{1}{2} - \rho\right)V_2 }) =  \f{\pi}{4}\f{\mu}{c} \f{s_1 s_2}{\tilde{k}^{\star}(s_1, s_2)}\left[ \f{1}{\cosh(\f{\pi}{2}\sqrt{\f{\mu}{c} s_1 - 1})} + \f{1}{\cosh(\f{\pi}{2}\sqrt{\f{\mu}{c} s_2 - 1})} \right],
\end{equation}
where $\tilde{k}^{\star}(s_1, s_2) = s_1   + s_2 + \frac{\mu}{8c}\big(s_1 - s_2\big)^2$.
\end{theorem}
\begin{proof}
By substituting $\tilde{\nu}^{\star}(s_j, s_j)$, $j = 1, 2$ (obtained from the LST \eqref{Chap4:LST_Total_worload}) into Equation \eqref{Scaling2HT1} we obtain $\tilde{\nu}^{\star}(s_1, s_2)$.
\end{proof}

\begin{remark}
Notice that letting $s_2 \to 0$ in \eqref{v^{star}(s_1, s_2)}, the r.h.s. tends to $\f{s_1}{\tilde{k}^{\star}(s_1, 0)} = \f{1}{1 +  \f{\mu}{8c} s_1}$, which is the heavy-traffic limit LST of the marginal workload
as given in Lemma~\ref{Chap4:MarginalHT}.
\end{remark} 

As a corollary we compute the first and second stationary moments of the joint workload in heavy traffic.
\begin{corollary}
\label{cor:Chap4:moments}
For $j = 1, 2$, it holds that
\begin{align*}
\E\left({\rm lim}_{\rho \uparrow 1/2} \left(\nicefrac{1}{2} - \rho\right) V_j\right) &=\frac{\mu}{8 c},  \\
\E\left({\rm lim}_{\rho \uparrow 1/2} \left(\nicefrac{1}{2} - \rho\right)^2 V_j^2\right) &=\frac{\mu^2}{32 c^2}, \\
 \E\left({\rm lim}_{\rho \uparrow 1/2}\left(\nicefrac{1}{2} - \rho\right)^2 V_1 V_2\right) &= \frac{\mu^2}{32 c^2} \f{\pi^2 - 9}{3},  \\
 \mathbb{R}\left({\rm lim}_{\rho \uparrow 1/2}\left(\left(\nicefrac{1}{2} - \rho\right)V_1, \left(\nicefrac{1}{2} - \rho\right)V_2\right) \right) &= \f{2}{3} \pi^2 - 7  \approx -0.4203,
\end{align*}
where $\mathbb{R}(\cdot, \cdot)$ is the correlation coefficient.
\end{corollary}
\begin{proof}
The marginal moments of the workload $V_j, j = 1, 2$ in heavy traffic are computed directly from Lemma \ref{Chap4:MarginalHT}. Equivalently,
Expression \eqref{Eq:ID_V1} can be used to compute the moments, namely
\begin{align*}
\E(\tilde{C})&=\sum_{n=1}^{\infty}\frac{1}{\frac{c}{\mu}\left((2n+1)^2-1\right)}=\frac{\mu}{4c},\\
\mathbb{V}\mathrm{ar}(\tilde{C})&=\sum_{n=1}^{\infty}\frac{1}{\frac{c^2}{\mu^2}\left((2n+1)^2-1\right)^2}=\frac{\mu^2}{16c^2}\frac{\pi^2-9}{3}
\end{align*}
with $\tilde{C}$ denoting the scaled total workload $\lim\limits_{\rho \uparrow \nicefrac{1}{2}} (\nicefrac{1}{2}- \rho)(V_1 +  V_2)$.
The joint  moment of $\lim\limits_{\rho \uparrow \nicefrac{1}{2}} (\nicefrac{1}{2} - \rho)^2 V_1 V_2$ is computed by differentiating the LST expression \eqref{v^{star}(s_1, s_2)} w.r.t. $s_1$ and $s_2$. 
%
\end{proof}

\section{Numerical results}\label{Chap4:Numerical_results}
In this section, we verify the obtained heavy-traffic results via simulations. Note that there are situations where simulation is not very efficient, and one such scenario appears in the heavy-traffic analysis of queueing models;
see, e.g., \cite{asmussen1992queueing}. Here it has been noted repeatedly that the standard simulation methods do not perform satisfactorily, one main problem being that the run lengths need to be exceedingly large to obtain even moderate precision. We have conducted simulations to validate our findings. One expects that as $\rho \uparrow \nicefrac{1}{2}$, the correlation coefficient tends to the exact correlation coefficient  $\mathbb{R}\left({\rm lim}_{\rho \uparrow 1/2}\left((\nicefrac{1}{2} - \rho)V_1,(\nicefrac{1}{2} - \rho)V_2\right)\right) = -0.4203$. For the parameters $c = 0.1$ and $\rho = 0.49$, we perform $1000$ batches of MaxTime ($2\times10^7$) simulations and calculate the correlation coefficient, the lower limit (LL), and the upper limit (UL) of the  $95\%$ confidence interval using the $1000$ samples of the correlation coefficients. The  runtime of each simulation takes approximately $2$ hours.

\begin{table}[h!]
\begin{center}
\scalebox{0.9}{
\begin{tabular}{ c | c | c |c}
~&Number of Runs $= 1000$, ~  MaxTime $ = 2 \times 10^7$   \\ 
\hline 
\begin{tabular}{c  }
$\rho$  \\
\hline
$0.2$  \\ $0.4$  \\  $0.47$ \\ $0.49$   \end{tabular} &  \begin{tabular}{c | c | c}
Confidence Interval LL & Simulated  $\mathbb{R}(V_1, V_2)$  & Confidence Interval UL  \\
\hline
$ -0.3954$ &  $-0.3954$ & $-0.3954$     \\
$ -0.4185$ &  $-0.4184$ & $-0.4184$ \\
$-0.4202$ &  $-0.4200$  & $-0.4200$     \\
$-0.4213$ &  $-0.4208$ & $-0.4202$
\end{tabular}
\end{tabular}}
\caption{Simulated correlation coefficient. The theoretical value for $\rho \to \nicefrac{1}{2}$   is $-0.4203$. \\ By properties of the correlation coefficient, $\mathbb{R}(V_1, V_2)$ equals $\mathbb{R}((\nicefrac{1}{2} - \rho)V_1, (\nicefrac{1}{2} - \rho)V_2)$.  }
\label{Chap4:T4}
\end{center}
\end{table}

From Table \ref{Chap4:T4}, we observe that as $\rho$ approaches $0.5$ from below, the simulation result approaches  $ \mathbb{R}\left({\rm lim}_{\rho \uparrow 1/2}\left((\nicefrac{1}{2} - \rho)V_1,(\nicefrac{1}{2} - \rho)V_2\right)\right) = -0.4203$. We also observe that the upper and lower limit of the confidence interval increase as $\rho$ approaches $\nicefrac{1}{2}$.

\begin{remark}
Notice from the simulation results  in Table \ref{Chap4:T4} that the correlation coefficient of the joint workload is not very sensitive to the traffic load.
\end{remark}

\begin{remark}
The scaled two-dimensional workload LST $\tilde{\nu}^{\star}(s_1, s_2)$ can be inverted numerically, cf. \cite{choudhury1994multidimensional, den2013wiener}. We have not been able to explicitly invert the LST. The scaled marginal distributions in heavy traffic are exponential (cf. Lemma \ref{Chap4:MarginalHT}), which suggests that the two-dimensional scaled workload distribution in heavy traffic might be a bivariate exponential distribution.   It is discussed in \cite[Theorem 4.2]{bladt2010construction} that the minimal correlation  of any bivariate exponential distributions is $1 - \nicefrac{\pi^2}{6} = -0.6449$. This does not exclude the possibility that the joint workload in heavy traffic has a bivariate exponential distribution, as our correlation equals $-0.4203$.
\end{remark}

For further validation of our heavy-traffic results, we plot the empirical cumulative distribution function (ECDF) of the scaled total workload in heavy traffic. For the parameters mentioned above, we  first compute the inverse Laplace transform of the expression given in \eqref{Chap4:LST_Total_worload} by using Talbot's method \cite{abate2006unified} in Matlab, then compare it with the simulation results.  The simulations are performed for the load $\rho = 0.2, 0.3, 0.4$ and $0.49$. Each simulation is performed for MaxTime ($1\times10^9$) which takes approximately $1$ hour. In Figure \ref{Chap4:fig:ECDF_jnvLST}, one can see that as $\rho$ approaches $0.49$ the simulation results also approach the results obtained from the empirical cumulative distribution computed numerically from the inverse LST of the expression given in \eqref{Chap4:LST_Total_worload}, i.e., ECDF\_jnvLST. 

\begin{figure}[h!]
\begin{center}
\includegraphics[width=8cm]{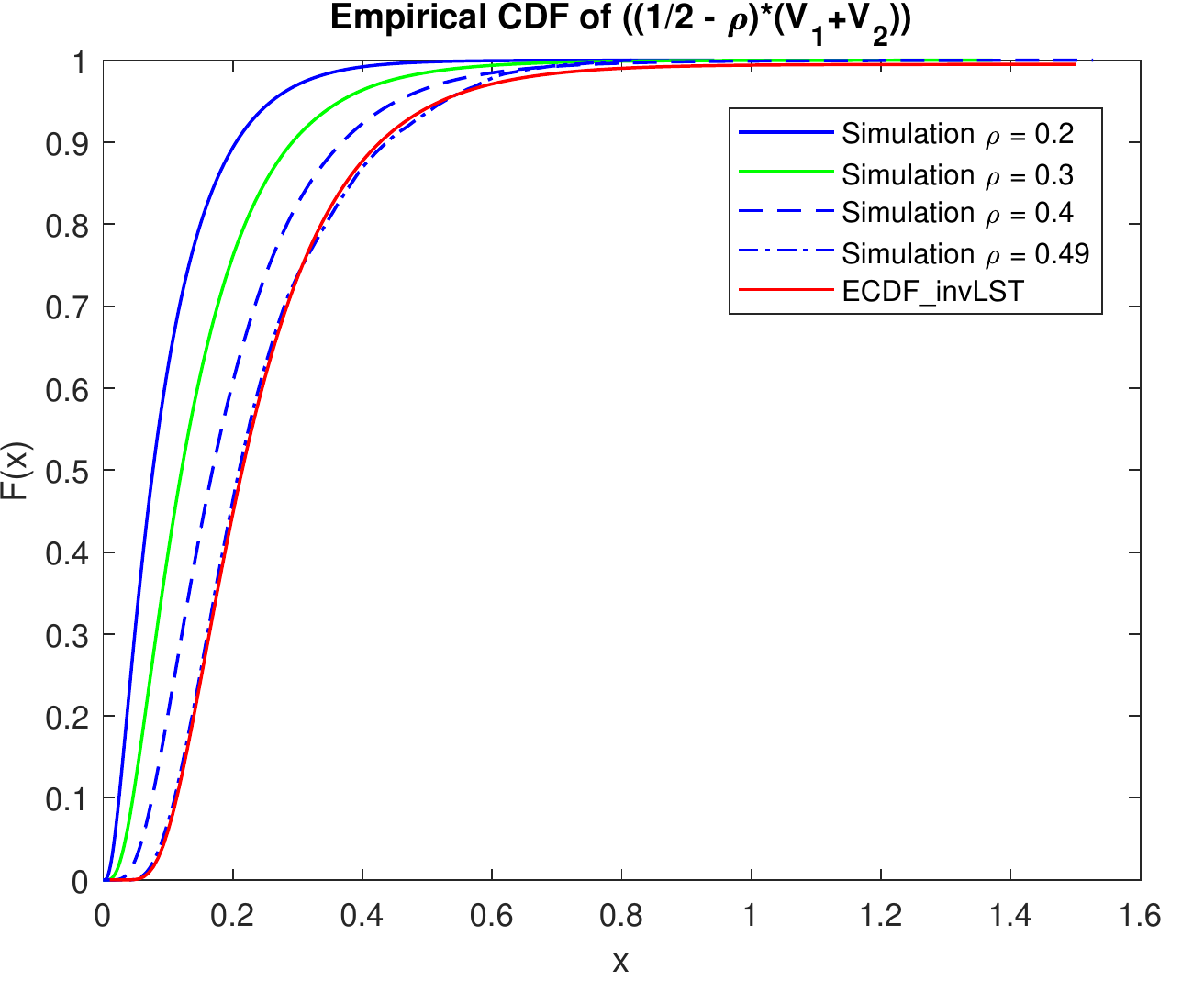}
\end{center}
\caption{Empirical cumulative distribution of the scaled total workload in heavy traffic.}
\label{Chap4:fig:ECDF_jnvLST}
\end{figure}

\section{Process limit in heavy traffic}\label{sec:processlimit}

Our main result so far in Theorem~\ref{thm:htjoint} established the heavy-traffic limit of the stationary joint distribution of the scaled workloads in the symmetric case.
In this section, we  investigate the heavy-traffic limit of the entire process of scaled workloads, under less restrictive assumptions on the input processes and the server switching process. We show that the stationary distribution of this limit process corroborates with the limit distribution of Theorem~\ref{thm:htjoint}, establishing that the time-stationary limit and the heavy-traffic limit can be interchanged.
Similar interchanges of limits have been previously demonstrated for different models in~\cite{gz2006} and~\cite{zz2008}.

As mentioned above, for the analysis in this section, we relax our assumptions regarding the input processes and the server switching process between the queues.
For the two queues, we assume Lévy subordinator inputs instead of constant fluid flows, and
the server visit periods form an alternating renewal process with possibly dependent consecutive visiting periods to server 1 and server 2. 
In the following, we make our assumptions precise.

We start with the server switching process: Specifically, we consider an i.i.d.\ sequence of nonnegative random pairs $\{(T_1(k),T_2(k)),\,k\ge 1\}$ distributed like $(T_1,T_2)$ where $\E(T_j)^2$, $j=1,2,$ are assumed finite (the marginal distributions of the $T_j$ are no longer assumed to be exponential). 
As before $\nicefrac{1}{c_j}=\E (T_j)$, and we denote $\sigma_j^2=\mathbb{V}\mathrm{ar}[T_j]$.
The covariance between consecutive visit periods to queue 1 and queue 2 is denoted by $\zeta=\mathbb{C}\mathrm{ov}(T_1,T_2)$.

Let $S_0=0$, $S_n=\sum_{k=1}^n(T_1(k)+T_2(k))$ for $n\ge 1$ and set $I(t)=1$ if $t\in\bigcup_{n=0}^\infty [S_n,S_n+T_1(n+1))$ and $I(t)=0$ otherwise. Assuming that $T_1+T_2$ is not almost surely (a.s.) zero, then with $p_1:=\nicefrac{c_2}{(c_1+c_2)}$ it is well known
that $\frac{1}{t}\int_0^tI(u){\rm d}u\to p_1$ a.s. and it is also known that
\begin{equation}
\frac{1}{\sqrt{n}}\int_0^{nt}(I(u)-p_1){\rm d}u
\end{equation}
converges weakly (in $D[0,\infty)$ endowed with the Skorohod $J_1$-topology) to a zero drift Brownian motion 
with variance given by
\begin{align}
\sigma^2&=\frac{\mathbb{V}\mathrm{ar}\left(\int_0^{T_1+T_2}I(u){\rm d}u-p_1(T_1+T_2)\right)}{\E(T_1+T_2)}
=\frac{\mathbb{V}\mathrm{ar}(T_1-p_1(T_1+T_2))}{\E(T_1+T_2)}\nonumber\\
&=\frac{c_1^2\sigma_1^2-2c_1c_2\zeta+c_2^2\sigma_2^2}{{(c_1+c_2)^3}}c_1c_2\ .
\end{align}
For a central limit theorem version of this, see \cite{takacs1959,hew2017}. This central limit version may also be concluded from \cite[Theorem. 3.2, page 178]{asmussen2003}. The functional limit theorem may be concluded from, e.g., \cite{gw1993,gw2002}. Let us denote this Brownian motion by $\sigma W(t)$ where $\{W(t),t\ge t\}$ denotes a Wiener process (standard Brownian motion).

Next we describe the input processes into the two queues, which we assume to be independent of the just described server switching process. We no longer assume that the input processes are constant fluid flows, but instead let the input into $Q_j$ be a Lévy process $\{J_j(t), t\geq0\}$, $j=1,2$.
To be precise, we assume that $\{J(t)\equiv (J_1(t),J_2(t)),\,t\ge 0\}$ is a bivariate subordinator with Laplace exponent $-\eta(s)$ where, for $(s_1,s_2) \in\mathbb{R}_+^2$,
\begin{align}
\eta(s_1,s_2)=b_1s_1+b_2s_2+\int_{\Rp^2}(1-e^{-s_1x_1-s_2x_2})\Pi({\rm d}x_1, {\rm d}x_2).
\end{align}
Here $(b_1,b_2)\in\mathbb{R}^2_+$ and $\Pi$ is the L\'evy measure  satisfying $\int_{\mathbb{R}_+^2}x_j\wedge 1\,\Pi({\rm d}x_1, {\rm d}x_2)<\infty$ for $j=1,2$. 
However, here we actually assume that $\int_{\mathbb{R}^2_+}x_j^2\Pi({\rm d}x_1, {\rm d}x_2)<\infty$, which is equivalent to the assumption that $\E (J_j^2(1))<\infty$ for $j=1,2$. 
Consistent with our earlier notation, for $j,i\in\{1,2\}$, we write
\begin{align}
\lambda_j&=\E(J_j(1))=b_j+\int_{\mathbb{R}_+^2}x_j\Pi({\rm d}x_1,{\rm d}x_2)=\frac{\partial\eta}{\partial s_j}(0+,0+);\nonumber\\
\sigma_{ji}&=\mathbb{C}\mathrm{ov}(J_j(1),J_i(1))=\int_{\mathbb{R}_+^2}x_jx_i\Pi({\rm d}x_1, {\rm d}x_2)
=\frac{-\partial^2\eta}{\partial s_j\partial s_i}(0+,0+)\ ,
\end{align}
and $\Sigma=(\sigma_{ji})_{j,i\in\{1,2\}}$ is the covariance matrix. Then $n^{-1/2}\left ( J_1(nt)-\lambda_1 nt,J_2(nt)-\lambda_2 nt\right)$ converges weakly to a zero mean ($2$-dimensional) Brownian motion with covariance matrix $\Sigma$. Let us denote this Brownian motion by $\{B(t)\equiv (B_1(t),B_2(t)),\,t\geq0\}$.
Having assumed that the processes $\{(T_1(k),T_2(k)),\,k\ge 1\}$ and $\{J(t), \, t \geq 0\}$ are independent, the Brownian motions $W$ (one-dimensional) and $B$ (two-dimensional) are independent as well.\\

We are now ready to describe the buffer content process of the two queues.
The cumulative input to $Q_j$ until time $t$ is $J_j(t)$, $j=1,2$. 
At instances where $I(t)=1$ (resp., $I(t)=0$) the server is working at $Q_1$ (resp., $Q_2$) at a rate of $\mu_1$ (resp., $\mu_2$). 
If we let $p_2=1-p_1$, and define the free processes
\begin{align}
X_1(t)&=J_1(t)-\mu_1\int_0^tI(u){\rm d}u\nonumber\\
&=J_1(t)-\lambda_1t+(\lambda_1-p_1\mu_1)t-\mu_1\int_0^t(I(u)-p_1){\rm d}u;\nonumber\\
X_2(t)&=J_2(t)-\mu_2\int_0^t(1-I(u)){\rm d}u\\
&=J_2(t)-\lambda_2t+(\lambda_2-p_2\mu_2)t+\mu_2\int_0^t(I(u)-p_1){\rm d}u,\nonumber
\end{align}
then the buffer content process associated with the $j$-th station ($j=1,2$) is given by the (continuous) functional
\begin{equation}
V_j(t)=X_j(t)-\inf_{0\le u\le t}X_j(u)\ .
\end{equation}
As is natural in our model, we assume that $\lambda_j>0$ for $j=1,2$ and that $0<p_1<1$. 
Let us replace $(\mu_1,\mu_2)$ by a sequence $(\mu_1^n,\mu_2^n)$, such that as $n\to\infty$,
\begin{equation}
\sqrt{n}(p_1\mu_1^n-\lambda_1,p_2\mu_2^n-\lambda_2)\to (\theta_1,\theta_2)\ .
\end{equation}
Although not necessary at this point, for later considerations we will assume that $\theta_j>0$, $j=1,2$.
For each value of $n$, $X_j^n(t)$ is the resulting free process with service rates $\mu_j^n$; and $V_j^n(t)=X_j^n(t)-\inf_{0\le u\le t}X_j^n(u)$ the corresponding buffer content process of $Q_j$, $j=1,2$.
Observing that $\mu_j^n\to \nicefrac{\lambda_j}{p_j}$, it follows that $n^{-1/2}X^n(nt)$ converges weakly to
\begin{align*}
X_1^\star(t)&=-\theta_1t+B_1(t)-\frac{\lambda_1\sigma}{p_1}W(t), \\ 
X_2^\star(t)&=-\theta_2t+B_2(t)+\frac{\lambda_2\sigma}{p_2}W(t).
\end{align*}
In particular, the covariance matrix of the limiting Brownian motion is given by
\begin{align}
\Sigma^\star&=
\begin{pmatrix}
\sigma_{11}+\frac{\lambda_1^2\sigma^2}{p_1^2}&\sigma_{12}-\frac{\lambda_1\lambda_2\sigma^2}{p_1p_2}\\
\\
\sigma_{12}-\frac{\lambda_1\lambda_2\sigma^2}{p_1p_2}&\sigma_{22}+\frac{\lambda_2^2\sigma^2}{p_2^2}
\end{pmatrix} .
\nonumber
\end{align}
By the continuous mapping theorem it also follows that $n^{-1/2}V^n(nt)$ converges weakly to $V^\star(t)$ with 
$V_j^\star(t)=X_j^\star(t)-\inf_{0\le u\le t}X_j^\star(u)$, $j=1,2$.\\

In the previous sections, we considered the special case $J_j(t)=\lambda_j t$, so that $\sigma_{ji}=0$ for $j,j=1,2$. If in addition $\sigma>0$ (note that this assumption only excludes the case in which $T_1(k)/T_2(k)$ is a fixed constant), we can define $\hat X_j^\star=\frac{p_j}{\lambda_j\sigma}X_j^\star$ and $\hat \theta_j=\frac{p_j}{\lambda_j\sigma}\theta_j$, $j=1,2$.
This results in
\begin{align}\label{eq:plusminus}
\hat X_1^\star(t)=-\hat \theta_1 t-W(t); ~~~
\hat X_2^\star(t)=-\hat \theta_2 t+W(t)\ .
\end{align}
Finally defining $\hat V_j^\star(t)=\hat X_j^\star(t)-\inf_{0\le u\le t}\hat X_j^\star(u)$, we observe that
$\hat V_j^\star(t)=\frac{p_j}{\lambda_j\sigma}V_j^\star(t)$, so that in order to study the stationary behavior of $V^\star(\cdot)$ it suffices to study that of $\hat V^\star(t)$. 
From now on it will be necessary that $\hat \theta_j>0$ for $j=1,2$, which is ensured by our earlier assumption that  $\theta_j>0$.

Let us first observe that for $s \in\mathbb{R}_+^2$ (actually for all $s \in\mathbb{R}^2$), from \eqref{eq:plusminus},  it follows after straightforward computations that
\[
\hat{k}(s_1,s_2)\equiv\log \E\left[e^{-s_1\hat X_1^\star(1)-s_2\hat X_2^\star(1)}\right]=\hat \theta_1 s_1+\hat \theta_2 s_2
+\frac{1}{2}(s_1-s_2)^2\ .
\]
With $\hat L_j^\star(t)=-\inf_{0\le u\le t}\hat X_j^\star(u)$, $j=1,2$, we know that $V_j^\star(t)=0$, for every point of (right) increase of $\hat L_j^\star(t)$. From this and the martingale of \cite{kw1992},
it may be concluded that the following is a zero mean martingale:
\begin{align}\label{Eq:Pre_FunctionalEquation_BivWorl}
\hat{k}(s_1,s_2)\int_0^t e^{-s_1\hat V_1^\star(u)-s_2\hat V_2^\star(u)}{\rm d}u&-e^{-s_1\hat V_1^\star(t)-s_2\hat V_2^\star(t)}+e^{-s_1\hat V_1^\star(0)-s_2\hat V_2^\star(0)}\nonumber
\\
&-s_1\int_0^t e^{-s_2 \tilde{V}_2^\star(u)}{\rm d}\hat L^\star_1(u)-s_2\int_0^te^{-s_1 \tilde{V}_1^\star(u)}{\rm d}\hat L_2^\star(u)\ .
\end{align}
It has become standard by now, see, e.g.,  \cite[Corollary 2.3]{k1993} (also from the theory of multivariate reflected Brownian motions on the nonnegative orthant), that if $\hat{v}^\star(s_1,s_2)$ is the LST of the stationary version of $\hat V^\star$, then taking expectations in Equation \eqref{Eq:Pre_FunctionalEquation_BivWorl} yields
\begin{align}
0&=t\hat{k}(s_1,s_2)\hat{\nu}^\star(s_1,s_2)-\hat{\nu}^\star(s_1,s_2)+\hat{\nu}^\star(s_1,s_2)\nonumber \\
&\quad -s_1\mathbb{E}\int_0^t e^{-s_2 \tilde{V}_2^\star(u)}{\rm d}\hat L^\star_1(u)
-s_2\mathbb{E}\int_0^t e^{-s_1 \tilde{V}_1^\star(u)}{\rm d}\hat L^\star_2(u)
\end{align}
and in particular for $t=1$ we have
\begin{equation}\label{Eq:Non_Sym_FunctionalEq}
\hat{k}(s_1,s_2)\hat{\nu}^\star(s_1,s_2)=s_1\hat{f}_1(s_2)+s_2 \hat{f}_2(s_1)\,,
\end{equation}
where  
\begin{align*}
\hat{f}_1(s_2)&=\mathbb{E}\int_0^1 e^{-s_2 \tilde{V}_2^\star(u)}{\rm d}\hat L^\star_1(u),\\
\hat{f}_2(s_1)&=\mathbb{E}\int_0^1 e^{-s_1 \tilde{V}_1^\star(u)}{\rm d}\hat L^\star_2(u)\,.
\end{align*}

Our objective is to determine the unknown function in the l.h.s.\ of \eqref{Eq:Non_Sym_FunctionalEq}: $\hat{\nu}^{\star}(s_1, s_2)$.

The present setting is in several respects much more general: a two-dimensional L\'evy input process, non-exponential visit periods, and asymmetry. 

In the symmetric case, viz. for $\hat{\theta}_1=\hat{\theta}_2=\hat{\theta}$, the key functional equation \eqref{Eq:Non_Sym_FunctionalEq} reduces to
\begin{equation}\label{Eq:Non_Sym_FunctionalEq_Symmetric}
\frac{s_1+ s_2+\frac{1}{2 \hat \theta}(s_1-s_2)^2}{s_1s_2}
\hat{\nu}^\star(s_1,s_2)=\frac{\hat{f}_1(s_2)}{s_2}+\frac{ \hat{f}_2(s_1)}{s_1},
\end{equation}
which is in essence identical to  \eqref{Scaling2HT1} for $\hat{\theta}=4c/\mu$. In this case, the starting point of the analysis matches, revealing that the results also match. It is important to note that, in  the symmetric case, although the analysis is identical to the one performed in Section  \ref{Chap4:HT_Analysis}, the setting of this section  is much broader than the one of Section \ref{Chap4:HT_Analysis}. 

In the analysis that follows, we do {\em not} restrict ourselves to a symmetric system as we did in Section \ref{Chap4:HT_Analysis}, instead we consider general $\theta_1$, $\theta_2$. For this general setting, we calculate the unknown function in the l.h.s.\ of \eqref{Eq:Non_Sym_FunctionalEq} using the BVM by applying Step A and Step B in an analogous manner as in Section \ref{Chap4:HT_Analysis}. Unfortunately, several of the convenient simplifications that transpire in the symmetric case and that  eventually led to the elegant result of Theorem \ref{HT_LST_Total_workload} are  not allowed in the asymmetric case $\theta_1\neq\theta_2$, as this can be seen in the analysis that follows and in the result of Theorem \ref{thm:htjoint_Levy_Model}. 

\paragraph{Kernel analysis.}
To apply the BVM, one needs to investigate the zeros of kernel $\hat{k}(s_1, s_2)$. By setting $\hat{k}(s_1, s_2) = 0$, we obtain

\begin{align}
\hat{s}_2^{\pm}(s_1)
&= s_1-\hat{\theta}_2\pm\sqrt{\hat{\theta}_2^2-2s_1( \hat{\theta}_1 + \hat{\theta}_2)}.
\end{align}
Note that $\hat{s}_2^{\pm}(s_1)$ has a single branching point at $s_1 = \nicefrac{ \hat{\theta}_2^2}{2( \hat{\theta}_1+ \hat{\theta}_2)}$. For real valued $s_1$ with $s_1 > \nicefrac{ \hat{\theta}_2^2}{2( \hat{\theta}_1+ \hat{\theta}_2)}$, the function $\hat{s}_2^{\pm}(s_1)$ is complex valued. Letting $\hat{s}_2^{\pm}(s_1) = u \pm i v$, we obtain, after straightforward computations, that 
\begin{align}\label{Parabola2}
v^2&= 2(\hat{\theta}_1+\hat{\theta}_2)\left(u+\frac{\hat{\theta}_2(2\hat{\theta}_1+\hat{\theta}_2)}{2(\hat{\theta}_1+\hat{\theta}_2)}\right),
\end{align}
which describes a parabola in the complex plane. We shall restrict ourselves to the following set:
$$\hat{E}_1  =  \set{(u, v) \in (-\nicefrac{\hat{\theta}_2(2\hat{\theta}_1+\hat{\theta}_2)}{2(\hat{\theta}_1+\hat{\theta}_2)}, \infty) \times \mathbb{R} \mid v^2 = 2(\hat{\theta}_1+\hat{\theta}_2)\left(u+\frac{\hat{\theta}_2(2\hat{\theta}_1+\hat{\theta}_2)}{2(\hat{\theta}_1+\hat{\theta}_2)}\right) }.$$
This domain will allow us to determine $\hat{f}_1(\cdot)$, while the symmetric domain obtained by considering the roots $\hat{s}_1^{\pm}(s_2)$ (which will result in a symmetric parabola with $\hat{\theta}_1$ and $\hat{\theta}_2$ interchanged) will allow us to determine  $\hat{f}_2(\cdot)$. 

\paragraph{BVM: Solution of the functional equation \eqref{Eq:Non_Sym_FunctionalEq}.}
Notice that, by definition, $\hat{f}_1(s_2)$ is analytic for $\mathrm{Re}[s_2] \geq 0$.
It still remains to show that $\hat{f}_1(s_2)$ is analytic on the strip $-\nicefrac{\hat{\theta}_2(2\hat{\theta}_1+\hat{\theta}_2)}{2(\hat{\theta}_1+\hat{\theta}_2)}< \mathrm{Re}[s_2] < 0$. 
We shall return to this point at a later stage, cf. Lemma \ref{Chap4:Claim2}.

Now we take $s_1$ with $s_1 > \nicefrac{ \hat{\theta}_2^2}{2( \hat{\theta}_1+ \hat{\theta}_2)}$ and $\hat{s}_2^{\pm}(s_1) = u \pm iv$, with $(u, v) \in \hat{E}_1$. For all such $(s_1, \hat{s}_2^{\pm}(s_1))$ pairs, the l.h.s. of \eqref{Eq:Non_Sym_FunctionalEq} becomes zero, and hence, for all $s_2 = \hat{s}_2^{\pm}(s_1)$, we have
\begin{equation}
\f{\hat{f}_1(s_2)}{s_2}=-\frac{ \hat{f}_2(s_1)}{s_1}.
\end{equation}
For $s_1 > \nicefrac{ \hat{\theta}_2^2}{2( \hat{\theta}_1+ \hat{\theta}_2)}$, the r.h.s. of the above equation is real, thus yielding
\begin{equation}\label{BVPJointHT2}
\mathrm{Re}\left[i\f{\hat{f}_1(s_2)}{s_2} \right] = 0, \, \, \mbox{for} \, \, s_2 =  \hat{s}_2^{\pm}(s_1) = u \pm iv, ~\mbox{with}~ (u, v) \in \hat{E}_1 \backslash \{(0, 0)\}.
\end{equation}

We thus see that $\nicefrac{\hat{f}_1(s_2)}{s_2} $ is analytic inside the contour $\hat{E}_1$, say $\hat{E}_1^+$, except for $s_2=0$ which is a pole in $\hat{E}_1^+$. The above problem now reduces to a Riemann-Hilbert problem with a pole, and with boundary $\hat{E}_1$,  see \cite[Section I.3.3]{cohen2000boundary}.  
To transform it into a (standard) Riemann-Hilbert problem on the unit circle $D$,
we define $\hat\phi_1$ (with inverse $\hat\psi_1$) to be a conformal mapping of the interior of the unit circle $D$ onto  $\hat{E}_1^+$ with normalization conditions 
$\hat{\phi}_1\left(-1\right)=\infty$, 
$\hat{\phi}_1(0)=\f{1 - \sqrt{2}\cos\left(\f{\pi \hat{\theta}_1}{2(\hat{\theta}_1+\hat{\theta}_2)} \right)}
{1 + \sqrt{2}\cos\left(\f{\pi \hat{\theta}_1}{2(\hat{\theta}_1+\hat{\theta}_2)} \right)}$, and  
$\hat{\phi}_1\left(1\right)=-\frac{\hat{\theta}_2(2\hat{\theta}_1+\hat{\theta}_2)}{2(\hat{\theta}_1+\hat{\theta}_2)}$. Following the same steps as in Section \ref{Chap4:HT_Analysis}, leading to Theorem \ref{HT_LST_Total_workload}, we again translate the Riemann-Hilbert BVP on and inside $\hat{E}_1$ to the simple Riemann-Hilbert BVP with a pole. The solution of the BVP 
\eqref{BVPJointHT2} is:
\begin{align}
\hat{f}_1(s_2) =\, &    \alpha_1 s_2- i \beta_1 \left(\hat{\psi}_1(s_2)-\f{1 - \sqrt{2}\cos\left(\f{\pi \hat{\theta}_1}{2(\hat{\theta}_1+\hat{\theta}_2)} \right)}
{1 + \sqrt{2}\cos\left(\f{\pi \hat{\theta}_1}{2(\hat{\theta}_1+\hat{\theta}_2)} \right)}\right)s_2 \nonumber\\
&+i \f{\bar{\beta} _1}{\hat{\psi}_1(s_2)-\f{1 - \sqrt{2}\cos\left(\f{\pi \hat{\theta}_1}{2(\hat{\theta}_1+\hat{\theta}_2)} \right)}
{1 + \sqrt{2}\cos\left(\f{\pi \hat{\theta}_1}{2(\hat{\theta}_1+\hat{\theta}_2)} \right)}} s_2, \, \, s_2 \in \hat{E}_1^+\cup\hat{E}_1\backslash\{0\},
\label{Eq:BVP:Levyf1}
\end{align}
where $\hat{\psi}_1(\cdot)$ is the conformal mapping from the parabola $\hat{E}_1$ to the unit circle $D$ given in the following lemma and the constants $ \alpha_1$ and $ \beta_1$, together with the full solution for the scaled buffer content processes are given in Theorem \ref{thm:htjoint_Levy_Model}.
\begin{lemma}\label{lemma:ConformalMap2}
For $z  \in \mathbb{C}$ and for $j=1,2$, the conformal map 
\begin{equation}\label{ConformalMap2}
\hat{\psi}_j(z) = 
\f{1 - \sqrt{2}\cosh\left(\f{\pi}{2(\hat{\theta}_1+\hat{\theta}_2)} \sqrt{2(\hat{\theta}_1+\hat{\theta}_2)z-{\hat{\theta}_j^2}}\right)}
{1 + \sqrt{2}\cosh\left(\f{\pi}{2(\hat{\theta}_1+\hat{\theta}_2)} \sqrt{2(\hat{\theta}_1+\hat{\theta}_2)z-\hat{\theta}_j^2}\right)}
\end{equation}
maps the interior of parabola 
$$v^2= 2(\hat{\theta}_1+\hat{\theta}_{2})\left(u+\frac{\hat{\theta}_{3-j}(2\hat{\theta}_j+\hat{\theta}_{3-j})}{2(\hat{\theta}_1+\hat{\theta}_2)}\right)$$ onto the interior of the unit circle $D$.
\end{lemma}
\begin{proof} The proof of the lemma is identical to that of Lemma \ref{lemma:ConformalMap2} and as such it is omitted.
\end{proof}

Now we are in position to state the  main theorem of this section, in which we obtain an explicit expression  for the scaled stationary buffer content process LST  in heavy traffic. 

\begin{theorem}\label{thm:htjoint_Levy_Model}
For $j=1,2$, the scaled stationary buffer content process LST  in heavy traffic is given by, for $\mathrm{Re}[s_j] > -\nicefrac{\hat{\theta}_j(2\hat{\theta}_{3-j}+\hat{\theta}_j)}{2(\hat{\theta}_1+\hat{\theta}_2)}$,
\begin{align}
\label{Sec5:LST_Total_worload}
\hat{f}_j(s_{3-j})&=\mathbb{E}\left[\int_0^1 e^{-s_{3-j} \hat{V}_{3-j}^\star(u)}{\rm d}\hat L^\star_j(u)\right]\nonumber\\
&= s_{3-j}  \frac{\pi  \sin \left(\frac{\pi  \hat{\theta} _j}{2 \left( \hat{\theta} _1+ \hat{\theta} _2\right)}\right) }{\left(\sqrt{2} \sin \left(\frac{\pi   \hat{\theta} _{3-j}}{2 \left( \hat{\theta} _1+ \hat{\theta} _2\right)}\right)+1\right)^2}
\Bigg[-\frac{\cos \left(\frac{\pi   \hat{\theta} _j}{ \hat{\theta} _1+ \hat{\theta} _2}\right)+4}{2 \sin \left(\frac{\pi   \hat{\theta} _{3-j}}{2 \left( \hat{\theta} _1+ \hat{\theta} _2\right)}\right)+\sqrt{2}}\nonumber\\
&\quad+
2 \sqrt{2} \Bigg(\frac{1}{\sqrt{2} \sin \left(\frac{\pi   \hat{\theta} _{3-j}}{2 \left( \hat{\theta} _1+ \hat{\theta} _2\right)}\right)+1}-\frac{1}{\sqrt{2} \cosh \left(\frac{\pi  \sqrt{2 \left( \hat{\theta} _1+ \hat{\theta} _2\right) s_{3-j} - \hat{\theta} _j^2}}{2 \left( \hat{\theta} _1+\theta _2\right)}\right)+1}\Bigg)\nonumber\\
&\quad+\frac{1}{\sqrt{2} \Bigg(\frac{1}{\sqrt{2} \sin \left(\frac{\pi   \hat{\theta} _{3-j}}{2 \left( \hat{\theta} _1+ \hat{\theta} _2\right)}\right)+1}-\frac{1}{\sqrt{2} \cosh \left(\frac{\pi  \sqrt{2 \left( \hat{\theta} _1+ \hat{\theta} _2\right) s_{3-j} - \hat{\theta} _j^2}}{2 \left( \hat{\theta} _1+ \hat{\theta} _2\right)}\right)+1}\Bigg)}-\sqrt{2}\Bigg]
.
\end{align}
For $j=1,2$, the scaled joint stationary buffer content process LST in heavy traffic is given by,  for $\mathrm{Re}[s_j] > -\nicefrac{\hat{\theta}_j(2\hat{\theta}_{3-j}+\hat{\theta}_j)}{2(\hat{\theta}_1+\hat{\theta}_2)}$,
\begin{align}
\label{v^{star}(s_1, s_2)_2}
\hat{\nu}^\star(s_1,s_2)&=\mathbb{E}\left[\int_0^1 e^{-s_1\hat V_1^\star(u)-s_2\hat V_2^\star(u)}{\rm d}u\right]=\nonumber\\
&= \frac{s_1s_2}{\hat{k}(s_1, s_2)}\left( \f{\hat{f}_1(s_2)}{s_2}+\f{\hat{f}_2(s_1)}{s_1}\right),
\end{align}
where $\hat{k}(s_1, s_2) = \hat \theta_1 s_1+\hat \theta_2 s_2
+\frac{1}{2}(s_1-s_2)^2$.
\end{theorem}

\begin{proof}
Setting $s_2=0$ yields on the one hand that the l.h.s. of \eqref{Eq:BVP:Levyf1} is equal to $\hat{f}_{1}(0)=\hat{\theta}_1$ and on the  other hand that the r.h.s. of \eqref{Eq:BVP:Levyf1}  is equal to $i \bar{{\beta}}_1 /\hat{\psi}_1'(0)$. Substituting $\hat{\psi}_1(z)$ from Lemma \ref{lemma:ConformalMap2} we obtain the value for ${\beta}_1$. Moreover, since $\hat{f}_{1}(\infty)=0$, we obtain the  value for ${\alpha}_1$. The same approach can also be used for the determination of $\hat{f}_2(s_1)$. After tedious, but straightforward computations, Equation \eqref{Sec5:LST_Total_worload} follows. \end{proof}

It is now convenient to formulate and prove the postponed Lemma~\ref{Chap4:Claim2}.
\begin{lemma}\label{Chap4:Claim2}
For $j=1,2$, the $j$-th scaled stationary buffer content process LST  in heavy traffic is analytic on the strip  $ -\nicefrac{\hat{\theta}_j(2\hat{\theta}_{3-j}+\hat{\theta}_j)}{2(\hat{\theta}_1+\hat{\theta}_2)}<\mathrm{Re}[s_j] <0$.
\end{lemma} 
\begin{proof}
Similar to the proof of Lemma \ref{Chap4:Claim}, we need to show that $\hat{f}_1(s_2)$ has no poles in  $ -\nicefrac{\hat{\theta}_j(2\hat{\theta}_{3-j}+\hat{\theta}_j)}{2(\hat{\theta}_1+\hat{\theta}_2)}<\mathrm{Re}[s_j] <0$. This is equivalent to considering the roots of the two denominators appearing in  Equation \eqref{Sec5:LST_Total_worload}, i.e., the zeros of $1+ \sqrt{2}\cosh\left(\f{\pi}{2(\hat{\theta}_1+\hat{\theta}_2)} \sqrt{2(\hat{\theta}_1+\hat{\theta}_2)s_{3-j}-\hat{\theta}_{j}^2}\right)$ and the zeros of
\begin{align*}
&\frac{1}{\sqrt{2} \sin \left(\frac{\pi   \hat{\theta} _{3-j}}{2 \left( \hat{\theta} _1+ \hat{\theta} _2\right)}\right)+1}-\frac{1}{\sqrt{2} \cosh \left(\frac{\pi  \sqrt{2 \left( \hat{\theta} _1+ \hat{\theta} _2\right) s_{3-j} - \hat{\theta} _j^2}}{2 \left( \hat{\theta} _1+ \hat{\theta} _2\right)}\right)+1}
.\end{align*}
For the former zeros, note that these are 
\begin{align}\label{zeros_Levy}
s_{3-j}=\frac{1}{2 \left( \hat{\theta} _1+ \hat{\theta} _2\right) }\left( 
\hat{\theta} _j^2-4(\hat{\theta} _1+ \hat{\theta} _2)^2(\nicefrac{3}{4}+2n)^2
\right),\ n\in\mathbb{Z}.
\end{align}
For the latter zeros, straightforward computations reveal that these are
\begin{align}\label{poles_Levy}
s_{3-j}=\frac{1}{2 \left( \hat{\theta} _1+ \hat{\theta} _2\right) }\left( 
\hat{\theta} _j^2-(-\hat{\theta}_j+4(\hat{\theta} _1+ \hat{\theta} _2)n)^2
\right),\ n\in\mathbb{Z}\setminus\{0\},
\end{align}
where we needed to exclude the case $s_{3-j}=0$ (which is equivalent to $n=0$ in the last expression), as this is not a pole for Equation \eqref{Sec5:LST_Total_worload}. In both cases,  \eqref{zeros_Levy} and \eqref{poles_Levy}, it is straightforward to show that $s_{3-j}<-\nicefrac{\hat{\theta}_j(2\hat{\theta}_{3-j}+\hat{\theta}_j)}{2(\hat{\theta}_1+\hat{\theta}_2)}$ for all $n$.
\end{proof}
Concluding this section, we would like to remark that in case $\hat{\theta}_1=\hat{\theta}_2=\hat{\theta}$, the result of Theorem \ref{thm:htjoint_Levy_Model} reduces exactly to that of Theorem \ref{thm:htjoint} for $\hat{\theta}=\nicefrac{4c}{\mu}$. This proves that the two limits (stationarity and heavy traffic) commute. Moreover, one can easily verify that, in the asymmetric case, taking the limit $\hat\theta_j\downarrow0$, while $\hat\theta_{3-j}>0$ yields
\begin{align*}
\lim_{\hat\theta_j\downarrow0}\hat{f}_j(s_{3-j})&=\lim_{\hat\theta_j\downarrow0}\mathbb{E}\left[\int_0^1 e^{-s_{3-j} \hat{V}_{3-j}^\star(u)}{\rm d}\hat L^\star_j(u)\right]=0,
\end{align*}
as the $\sin \left(\frac{\pi   \hat{\theta} _j}{2 \left( \hat{\theta} _1+ \hat{\theta} _2\right)}\right)$ becomes zero and all other quantities are bounded.

\section*{Acknowledgements }
The  authors  gratefully acknowledge useful discussions with Sindo N\'u\~nez Queija from the Korteweg-de Vries Institute at the University of Amsterdam, who provided insight and expertise during the course of this research. 

The work of Stella Kapodistria and Onno Boxma is supported by the Netherlands Organisation for Scientific Research (NWO) through Gravitation-grant NETWORKS-024.002.003. The research of Mayank Saxena was funded by the NWO TOP-C1 project of the Netherlands Organisation for Scientific Research. Offer Kella is supported by grant No. 1647/17 from the Israel Science Foundation and the Vigevani Chair in Statistics.

\end{document}